\documentclass[11pt,reqno]{amsart}
\usepackage[top=3cm,bottom=3cm,inner=3.5cm,outer=3.5cm,includehead,includefoot]{geometry}
\usepackage[T1]{fontenc}
\usepackage{amsmath}
\usepackage{amssymb}
\usepackage{amsthm}
\usepackage{calc}
\usepackage{enumitem}
\usepackage{mathrsfs}
\usepackage{paralist}
\usepackage{tikz}
\usetikzlibrary{shapes.misc,calc,intersections,patterns}
\usepackage[ocgcolorlinks]{hyperref}

\numberwithin{equation}{section}

\newtheorem{theorem}{Theorem}[section]
\newtheorem{lemma}[theorem]{Lemma}

\newtheorem{claim}[theorem]{Claim}
\newtheorem{corollary}[theorem]{Corollary}
\newtheorem{conjecture}[theorem]{Conjecture}

\theoremstyle{remark}
\newtheorem{remark}[theorem]{Remark}

\newenvironment{enumerate*}{\begin{compactenum}[(i)]}{\end{compactenum}}
\def\E{\mathbb{E}}
\def\N{\mathbb{N}}
\def\P{\mathbb{P}}
\def\R{\mathbb{R}}
\def\T{\mathbb{T}}
\def\Z{\mathbb{Z}}
\renewcommand{\leq}{\leqslant}
\renewcommand{\geq}{\geqslant}

\newcommand{\abs}[1]{\lvert #1 \rvert}
\newcommand{\norm}[1]{\lVert #1 \rVert}
\newcommand{\pos}{\mathcal{P}}
\newcommand{\NN}{\mathcal{N}}
\newcommand{\free}{\mathcal{F}}
\newcommand{\configs}{\mathcal{C}_d}
\DeclareMathOperator{\Po}{Po}
\DeclareMathOperator{\dTV}{d_{TV}}

\title{The time of bootstrap percolation with dense initial sets}

\author[B. Bollob\'as]{B\'ela Bollob\'as}
\address{Trinity College, Cambridge, CB2 1TQ, UK, and Department of Mathematical Sciences, University of Memphis, Memphis, TN 38152, USA}
\email{\href{mailto:b.bollobas@dpmms.cam.ac.uk}{b.bollobas@dpmms.cam.ac.uk}}

\author[C. Holmgren]{Cecilia Holmgren}
\address{Department of Pure Mathematics and Mathematical Statistics, Wilberforce Road, Cambridge, CB3 0WA, UK}
\email{\href{mailto:c.holmgren@dpmms.cam.ac.uk}{c.holmgren@dpmms.cam.ac.uk}}

\author[P.J. Smith]{Paul Smith}
\email{\href{mailto:p.j.smith@dpmms.cam.ac.uk}{p.j.smith@dpmms.cam.ac.uk}}

\author[A.J. Uzzell]{Andrew J. Uzzell}
\address{Department of Mathematical Sciences, University of Memphis, Memphis, TN 38152, USA}
\email{\href{mailto:ajuzzell@memphis.edu}{ajuzzell@memphis.edu}}

\date{\today}

\thanks{The first and fourth authors were partially supported by ARO grant W911NF-06-1-0076 and by NSF grant DMS-0906634.  The second author was supported by a grant from the Swedish Research Council.  The second and third authors are grateful to the University of Memphis, where some of this research was carried out.  The fourth author is grateful to the University of Cambridge, where some of this research was carried out.}

\subjclass[2010]{Primary 60K35; Secondary 60C05}
\keywords{Bootstrap percolation, sharp threshold, Stein-Chen method}

\begin{document}

\begin{abstract}
In $r$-neighbour bootstrap percolation on the vertex set of a graph~$G$, vertices are initially infected independently with some probability~$p$.  At each time step, the infected set expands by infecting all uninfected vertices that have at least~$r$ infected neighbours.  When $p$ is close to~1, we study the distribution of the time at which all vertices become infected.  Given $t = t(n) = o(\log n/\log\log n)$, we prove a sharp threshold result for the probability that percolation occurs by time~$t$ in $d$-neighbour bootstrap percolation on the $d$-dimensional discrete torus~$\T_n^d$.  Moreover, we show that for certain ranges of~$p = p(n)$, the time at which percolation occurs is concentrated either on a single value or on two consecutive values.  We also prove corresponding results for the modified $d$-neighbour rule.
\end{abstract}

\maketitle

\section{Introduction}\label{se:fastintro}

Bootstrap percolation is an example of a cellular automaton, a concept developed by von~Neumann~\cite{vN} following a suggestion of Ulam~\cite{Ulam}.  Bootstrap percolation was introduced by Chalupa, Leath, and Reich~\cite{ChaLeathReich} in the context of the Blume-Capel model of ferromagnetism. In bootstrap percolation on the vertex set of a graph~$G$, vertices have two possible states, `infected' and `uninfected'.  Let $r \in \N$, let $G$ be a locally finite graph, and let $A \subset V(G)$ denote the set of initially infected vertices.  In this paper, as often, elements of~$A$ are chosen independently at random with some probability~$p$.  In \emph{$r$-neighbour bootstrap percolation}, infected vertices remain infected, and if an uninfected vertex has at least~$r$ infected neighbours, then it becomes infected.  Formally, setting $A_0 = A$ and letting $N(v)$ denote the neighbourhood of~$v$, we have
\[
A_{t+1} = A_t \cup \{v \, : \, \abs{N(v) \cap A_t} \geq r\}
\]
for all~$t \geq 0$.  If, for some $t$, we have $A_t = V(G)$, we say that $A$ \emph{percolates $G$}, or simply that $A$ \emph{percolates}.

Van~Enter~\cite{vanEnter} and Schonmann~\cite{Schon} showed that for $G = \Z^d$ and $p \in (0, 1)$, under the standard $r$-neighbour model, if $r\leq d$, then percolation almost surely occurs; while if $r \geq d + 1$, then percolation almost surely does not occur.

In the case of $r$-neighbour bootstrap percolation on the $d$-dimensional grid~$[n]^d$, where $d \geq r \geq 2$, the probability of percolation displays a \emph{sharp threshold}.  That is, there exists a value~$p_c = p_c(n)$ such that for all~$\varepsilon > 0$, if $p < (1 - \varepsilon)p_c$, then the probability of percolation is close to~0, while if $p > (1 + \varepsilon)p_c$, then the probability of percolation is close to~1.
Models for which sharp thresholds are now known to exist include $r$-neighbour bootstrap percolation on $[n]^d$, for every $2\leq r\leq d$ (see~\cite{BBDCM,BBM3D,GravHol,GravHolMor,Hol}); various other update rules on $\Z^2$ (see~\cite{HolModAllD,BringMahl,BringMahlMel,DCvE,DCHol,HLR}); and two-neighbour percolation on the hypercube~$\{0, 1\}^n$ (see~\cite{BBHyp}).

We note that Balogh and Bollob\'as~\cite{BBSharp} studied a different notion of sharp threshold for two-neighbour bootstrap percolation on $[n]^d$.  With the threshold~$r$ implicit, set
\[
P(G, \alpha) = \inf\{p : \P_p(G \text{ percolates in $r$-neighbour bootstrap percolation}) \geq \alpha\}.
\]
Balogh and Bollob\'as showed that for any~$\varepsilon > 0$, $P([n]^d, 1 - \varepsilon) - P([n]^d, \varepsilon) = o\bigl(p_c(n)\bigr)$.

In the case of~$\Z^d$, the probability that the initially infected set~$A$ percolates $[n]^d$ turns out to be closely related to the probability that the origin becomes infected by time~$n$ if the process is run on $\Z^d$.  Set $T_0 = \min\{t: 0 \in A_t\}$. Andjel, Mountford, and Schonmann~\cite{Andjel,AMS,Mountford,Schon} proved sharp results about the limiting behaviour of the probability that $T_0$ is at least some fixed~$t$.

In bootstrap percolation, extremal results are often important for proving probabilistic results.  At first, this may seem surprising, but in fact, it is quite natural.  The reason that extremal results are important is that the only randomness in the process occurs in the initial infection process.  Consequently, in proving results about bootstrap percolation, much of the work often involves analysing the deterministic evolution of an arbitrary initial configuration.

One of the first extremal results in bootstrap percolation was a result of Morris~\cite{MorrisMin} on the largest size of a minimal set that percolates in $[n]^2$.  Later, Riedl~\cite{RiedlHyp} continued this work in the case of standard two-neighbour percolation on the hypercube. Riedl~\cite{RiedlTree} also gave bounds on the sizes of the largest and smallest minimal percolating sets for $r$-neighbour percolation in the case when $r \geq 2$ and $G$ is a tree on $n$ vertices with $\ell$ vertices of degree less than~$r$.  The first extremal result on the time of percolation was a theorem of Benevides and Przykucki~\cite{BenPrz} that answered the extremal question of finding the maximum percolating time on an $a\times b$ rectangular grid. If $A$ percolates an $a \times b$ grid, it is not hard to show that $\abs{A} \geq \lceil (a + b)/2 \rceil$ (see \cite{BalPete} or~\cite{coffeetime}).  When $a=b=n$, Benevides and Przykucki proved that if $A$ is a percolating set of size exactly~$n$, then $A$ percolates in time at most~$\tfrac{5}{8}n^2+O(n)$, while if $A$ is any percolating set, then $A$ percolates in time at most~$\tfrac{13}{18}n^2+O(n)$. They proved that both bounds are tight.  Przykucki~\cite{PrzHyp} proved corresponding results for two-neighbour percolation on the hypercube.

If $A$ percolates $V(G)$, we define the \emph{time of percolation} or \emph{percolation time} to be
\[
T := T(G; A) := \min\{t \, : \, A_t = V(G)\}.
\]
In the probabilistic setting, perhaps the most natural question that one could ask about the time of percolation is the following: given a bootstrap process and an initial probability such that percolation occurs with high probability, what is the percolation time~$T$?  In this paper we give a complete answer to this question in the case of $d$-neighbour bootstrap percolation on the discrete $d$-dimensional torus, when the percolation time~$T$ is small (or, equivalently, when the probability~$p$ is close to~$1$).  In order to answer this probabilistic question, we shall need to prove several extremal results about sets that do not percolate within a given time.

In~\cite{JLTV}, Janson, {\L}uczak, Turova, and Vallier determined the asymptotic time of percolation for $r$-neighbour percolation on an Erd\H{o}s-R\'enyi random graph.  To the best of our knowledge, this question has not been otherwise studied.

Our main aim is to show that with high probability the percolation time~$T$ is in a certain small interval.  To that end, our main task will be to show that for any not-too-large value of~$t$, the number of uninfected vertices at time~$t$ is asymptotically Poisson distributed. It will follow that the probability that $T$ is at most~$t$ is asymptotically the probability that a Poisson random variable equals~0.  To prove Poisson convergence, we use the Stein-Chen method~\cite{Chen,Stein}, a tool often applied to prove convergence in distribution. The power of the Stein-Chen method is that it only requires knowledge of the first two moments of the distribution for which we are trying to prove convergence. Obtaining good bounds on these first two moments occupies the majority of this paper.

In the literature of percolation theory, it is common to refer to the vertices of a graph as `sites'.  In this paper, we shall use the terms `vertex' and `site' interchangeably.

\begin{remark}\label{re:ALPoisson}
Aizenman and Lebowitz~\cite{AizLeb} observed that in bootstrap percolation on $[n]^d$, the event that the infected set percolates depends on the formation of a ``critical droplet'', that is, of an infected cube of side length on the order of~$\log n$.  They also observed that for $n$ not too large compared to~$p$, the events that different cubes of this size become fully infected are nearly independent. This adds weight to the hypothesis that the behaviour of the number of uninfected sites should be approximately Poisson distributed.
\end{remark}

Our main tool in proving the sharp threshold result for the time of percolation is the solution to an extremal problem that may be of independent interest.  Namely, we wish to determine the maximum size of a set that does not infect a given site (which we can assume is the origin) by time~$t$. Equivalently, we would like to determine, for all $t\geq 1$ and~$d\geq 2$, the function
\begin{equation}\label{eq:extremalq}
\operatorname{ex}(t, d) := \min_{A\subset\Z^d} \bigl\{\lvert \Z^d\setminus A \rvert \, : \, 0\notin A_t\bigr\},
\end{equation}
where, as before, $A$ denotes the set of initially infected sites.

Which configurations of uninfected sites guarantee that the origin is uninfected at time~$t$? It is easy to see that the event that the origin is uninfected at time~$t$ is independent of the states of sites at $\ell_1$ distance greater than~$t$ from the origin. One such configuration is an empty $\ell_1$ ball of radius~$t$ about the origin. Another configuration, with far fewer sites than the whole $\ell_1$ ball, is a set of the form
\begin{equation}\label{eq:column}
C_t := \bigl\{x = (\varepsilon_1,\dots,\varepsilon_{d-1},r) \, : \, \lVert x\rVert_1\leq t \text{ and } \varepsilon_i\in\{0,1\} \text{ for } i\in[d-1]\bigr\};
\end{equation}
we may think of this set as a column vertically centred at the origin. In fact, we prove that the minimum quantity~\eqref{eq:extremalq} is equal to the size of this set; moreover, we show that that these columns are essentially the only sets that both achieve the minimum and guarantee that $0 \notin A_t$.

The following quantity, which is the size of the set in \eqref{eq:column}, is the number of sites with $\ell_1$ norm at most~$t$ in a column centred at the origin. Set
\begin{equation}\label{eq:minball}
m_t := m _{t,d} := \abs{C_t} = \sum_{r=0}^t \Bigg( 2\sum_{j=0}^{r-1} \binom{d}{j} + \binom{d}{r} \Bigg) = \sum_{r = 0}^t \sum_{j = 0}^r \dbinom{d}{j}.
\end{equation}
(We follow the convention that $\binom{d}{j} = 0$ whenever $j>d$.) We shall show that $\operatorname{ex}(t,d) = m_{t,d}$.  In fact, we shall prove a more general result about the minimum number of uninfected vertices that are distance~$k$ from a vertex~$x$ that remains uninfected for a long time.

\begin{remark}\label{re:anisotropic}
While studying thresholds for percolation for certain anisotropic bootstrap percolation models, van~Enter and Hulshof~\cite{vEH} and Mountford~\cite{MountfordSemi} studied the event that a single or double column is \emph{full}.
\end{remark}

These extremal results allow us to determine bounds on both the mean and the variance of the distribution of the number of sites that are uninfected at time~$t$; in turn, this enables us to prove Poisson convergence for values of~$t = t(n)$ up to~$o(\log\log n$). In order to extend this range of~$t$ to~$o(\log n/\log\log n)$, we need much stronger bounds on the mean and variance. The mean is proportional to the probability~$p_1$ that a given site (which we can assume is the origin) is uninfected at time~$t$. Thus we can express $p_1$ in terms of the number of initially uninfected sites inside the $\ell_1$ ball of radius~$t$. The results described above bound the first (and dominant) term in this expansion, namely, the one corresponding to the minimum number of uninfected sites such that the origin is still uninfected at time~$t$. In order to bound the mean and variance of our distribution, we need not just bounds on the highest-order term in the expansion of~$p_1$, but also on all of the other terms. In other words, we need to understand the number of configurations of uninfected sites when the number of uninfected sites preventing the origin from becoming infected by time~$t$ is a just a few more than the minimum. Roughly, we prove that if the number of uninfected sites is not much more than the minimum, then the configuration is close to an extremal configuration.  This stability result gives stronger bounds on the first two moments of our distribution.

Before we can state our main results, we need to formalize our notation. The discrete $d$-dimensional torus~$\T_n^d$ is the graph with vertex set~$(\Z/n\Z)^d$ in which vertices are adjacent if and only if their $\ell_1$ distance is exactly~$1$. As usual, let $A$ be a random subset of~$\T_n^d$ in which vertices are infected independently with probability~$p$, and let $\P_p$ be the associated product probability measure. Let $T = T(\T_n^d)$. Given $t\in\N$ and~$\alpha \in [0, 1]$, we set
\begin{equation}\label{eq:critprob}
p_{\alpha}(t) := \inf\{p \, : \, \P_p(T \leq t)\geq \alpha\}.
\end{equation}
In sharp threshold results for the probability of percolation, it is common to define $\inf\{p : \P_p(A \text{ percolates}) \geq 1/2\}$ to be the \emph{critical probability}, denoted $p_c$, and to show that the probability of percolation displays a sharp threshold at $p_c$.  However, in the proofs of these results, the choice of~$1/2$ is irrelevant: it turns out that for any constant~$\alpha \in (0, 1)$, $p_{\alpha} = \bigl(1 + o(1)\bigr)p_{1/2}$.  In our case, $p_{\alpha}(t)$ is different for different values of $\alpha \in (0, 1)$.

The first main result of this paper is as follows.  As usual, given~$p \in [0, 1]$, we write $q = 1 - p$.
\begin{theorem}\label{th:crittime}
Let $d\geq 2$, let $t = o(\log n/\log\log n)$, let $(p_n)_{n=1}^\infty$ be a sequence of probabilities, let $\omega(n) \to \infty$, and let $T = T(\T_n^d)$.  Under the standard $d$-neighbour model,
\begin{enumerate*}
\item if, for all $n$, $q_n \leq \bigl(n^{-d}\!/\omega(n)\bigr)^{1/m_{t}}$, then $\P_{p_n}(T \leq t) \to 1$ as $n\to\infty$;
\item if, for all $n$, $q_n \geq \bigl(n^{-d} \omega(n)\bigr)^{1/m_{t}}$, then $\P_{p_n}(T \leq t) \to 0$ as $n\to\infty$.
\end{enumerate*}
Moreover, for any $\alpha \in (0, 1)$,
\[
p_{\alpha}(t) = 1 - \bigl(1+o(1)\bigr)\biggl(\dfrac{\log\left(\frac{1}{\alpha}\right)}{d^3 2^{d-1} n^d}\biggr)^{\frac{1}{m_{t,d}}}.
\]
\end{theorem}

The analogue of Theorem \ref{th:crittime} holds in the case of the \emph{modified $d$-neighbour bootstrap percolation} model. In this process, an uninfected vertex becomes infected if it has at least one infected neighbour in each direction, that is, for all~$t \geq 0$,
\[
A_{t+1} = A_t \cup \big\{v \, : \, \text{for all } i \in [d], \, \big\lvert A_t \cap \{v - e_i, v + e_i\} \big\rvert \geq 1 \big\},
\]
where $e_i$ denotes the $i$th standard basis vector in $\R^d$. Let $p_\alpha^{(m)}(t)$ be the quantity in the modified $d$-neighbour model corresponding to $p_\alpha(t)$ in the standard $d$-neighbour model, as defined in \eqref{eq:critprob}.

\begin{theorem}\label{th:crittimemod}
Let $d\geq 2$, let $t = o(\log n/\log\log n)$, let $(p_n)_{n=1}^\infty$ be a sequence of probabilities, let $\omega(n) \to \infty$, and let $T = T(\T_n^d)$.  Under the modified $d$-neighbour model,
\begin{enumerate*}
\item if, for all $n$, $q_n \leq \bigl(n^{-d}\!/\omega(n)\bigr)^{1/(2t + 1)}$, then $\P_{p_n}(T \leq t) \to 1$ as $n\to\infty$;
\item if, for all $n$, $q_n \geq \bigl(n^{-d} \omega(n)\bigr)^{1/(2t + 1)}$, then $\P_{p_n}(T \leq t) \to 0$ as $n\to\infty$.
\end{enumerate*}
Moreover, for any $\alpha \in (0, 1)$,
\[
p_{\alpha}^{(m)}(t) = 1 - \bigl(1+o(1)\bigr)\biggl(\dfrac{\log\left(\frac{1}{\alpha}\right)}{d n^d}\biggr)^{\frac{1}{2t + 1}}.
\]
\end{theorem}

The proof of Theorem~\ref{th:crittimemod} follows the same structure as the proof of Theorem~\ref{th:crittime} but is vastly simpler. While the deduction of the Poisson convergence result from the combinatorial results is essentially the same in either case, the proofs of the combinatorial results, which form the backbone of the proof of Theorem~\ref{th:crittime}, are trivial in the case of the modified $d$-neighbour model. Therefore, the proof of Theorem~\ref{th:crittimemod}, which appears in Section \ref{se:mod}, is only sketched.

\begin{remark}
Observe that the discrete torus~$\T_n^d$ is a vertex-transitive graph, which means that the $\ell_1$ balls of radius~$t$ around different vertices are identical.  This makes the discrete torus a natural setting in which to consider the problem of percolation by time~$t$.
\end{remark}

\begin{remark}
It is important to note that fast percolation in the case of a high infection probability is very different to the last few steps of near-to-critical percolation, when the probability~$p$ is just above the critical probability for percolation. In the former case, the initial set~$A$ consists of sites which are infected independently at random with probability~$p$, which is close to~$1$, while in the latter case, if percolation occurs at time~$T$, then for small values of~$t$, the set~$A_{T-t}$ consists of sites which are far from independently infected: as shown by Aizenman and Lebowitz~\cite{AizLeb}, with high probability, $A_{T-t}$ will consist of one large rectangle covering almost the entire domain, and just a few additional sites.
\end{remark}

One of the strengths of Theorem~\ref{th:crittime} is that it allows us to deduce that if $t = o(\log n/\log\log n)$ and $q_n$ is bounded away from both $n^{-d/m_{t-1}}$~and~$n^{-d/m_{t}}$, then, with high probability, $T = t$, and otherwise, there is a two-point concentration for $T$. The following theorem is our second main theorem.

\begin{theorem}\label{th:conc}
Let $d \geq 2$, let $t = o(\log n/\log\log n)$, and let $(p_n)_{n=1}^{\infty}$ be a sequence of probabilities.  Consider the standard $d$-neighbour rule.
\begin{enumerate*}
\item Suppose that there exists $\omega(n) \to \infty$ such that
\begin{equation}\label{eq:1ptrange}
\bigl(n^{-d} \omega(n)\bigr)^{1/m_{t-1}} \leq q_n \leq \bigl(n^{-d}\!/\omega(n)\bigr)^{1/m_{t}}.
\end{equation}
Then, with high probability, $T = t$.
\item Suppose instead that
\begin{equation}\label{eq:2ptrange}
\bigl(n^{-d}\!/\omega(n)\bigr)^{1/m_t} \leq q_n \leq \bigl(n^{-d} \omega(n)\bigr)^{1/m_t}
\end{equation}
for all functions~$\omega(n) \to \infty$.  Then, with high probability, $T \in \{t, t + 1\}$.  Moreover, if there exists a constant~$c$ such that $\lim_{n \to \infty} q_n^{m_t} n^d = c$, then
\[
\P_{p_n}(T = t) \sim 1 - \P_{p_n}(T = t + 1) \sim \exp\bigl(-d^3 2^{d-1} c\bigr).
\]
\end{enumerate*}
\end{theorem}

Again, we have a corresponding result for the modified $d$-neighbour rule.

\begin{theorem}\label{th:concmod}
Let $d \geq 2$, let $t = o(\log n/\log\log n)$, and let $(p_n)_{n=1}^{\infty}$ be a sequence of probabilities.  Consider the modified $d$-neighbour rule.
\begin{enumerate*}
\item Suppose that there exists $\omega(n) \to \infty$ such that
\[
\bigl(n^{-d} \omega(n)\bigr)^{1/(2t - 1)} \leq q_n \leq \bigl(n^{-d}\!/\omega(n)\bigr)^{1/(2t + 1)}.
\]
Then, with high probability, $T = t$.
\item Suppose instead that
\[
\bigl(n^{-d}\!/\omega(n)\bigr)^{1/(2t + 1)} \leq q_n \leq \bigl(n^{-d} \omega(n)\bigr)^{1/(2t + 1)}
\]
for all functions~$\omega(n) \to \infty$.  Then, with high probability, $T \in \{t, t + 1\}$.  Moreover, if there exists a constant~$c$ such that $\lim_{n \to \infty} q_n^{2t + 1} n^d = c$, then
\[
\P_{p_n}(T = t) \sim 1 - \P_{p_n}(T = t + 1) \sim \exp(-d c).
\]
\end{enumerate*}
\end{theorem}

Once again, the proof of Theorem~\ref{th:concmod} is very similar to that of Theorem~\ref{th:conc}, so we shall omit it.

The rest of the paper is organized as follows. In Section~\ref{se:probtools}, we recall important terminology from probability theory and introduce the Stein-Chen method for proving convergence in distribution to a Poisson random variable. In Section~\ref{se:extremal}, we study the extremal questions connected with a vertex being uninfected at time~$t$. These fall into two categories. First, in Section~\ref{se:exact}, we answer completely the exact questions: what is the minimum number of uninfected sites needed to ensure a given site is uninfected at time~$t$, and what are the minimal configurations? Second, in Section~\ref{se:stability}, we look at the inexact questions: what can we say about the number and type of configurations when the number of uninfected sites is not much more than minimum number? We show that the set of uninfected sites must still be quite close to a column. In Section~\ref{se:proofs}, we put together the probabilistic tools from Section~\ref{se:probtools} and the extremal results from Section~\ref{se:extremal} to prove Theorems \ref{th:crittime} and~\ref{th:conc}. In Section~\ref{se:mod}, we sketch the proof of Theorem~\ref{th:crittimemod}.  Finally, in Section~\ref{se:open}, we discuss possible generalizations and conjectures.

\section{The Stein-Chen method}\label{se:probtools}

In this section, we recall the tools and techniques from probability theory that we need in the proof of Theorem~\ref{th:crittime}.

For a random variable $X$, we write $X\sim\Po(\lambda)$ to indicate that $X$ has Poisson distribution with mean~$\lambda$.

Let $P$ and $Q$ be probability distributions with support on $\Z$. The \emph{total variation distance} of $P$~and~$Q$ is
\[
\dTV(P,Q) = \sup_{A\subset\Z} \bigl\lvert \P (X\in A) - \P (Y\in A) \bigr\rvert.
\]
If $X$ and $Y$ are random variables with distributions $P$ and~$Q$ respectively, then with a slight abuse of notation we write $\dTV(X,Y)$ for $\dTV(P,Q)$. Let $(X_n)_{n=1}^{\infty}$, $(Y_n)_{n=1}^{\infty}$ be sequences of integer-valued random variables. We say that the sequences $(X_n)_{n=1}^{\infty}$ and~$(Y_n)_{n=1}^{\infty}$ \emph{converge in distribution} if $\lim_{n\rightarrow \infty} \dTV (X_n, Y_n) = 0$.

In order to prove Theorem~\ref{th:crittime} we need to show that a certain sequence of random variables converges to the Poisson distribution. Classically, to prove convergence in distribution one had to use the method of moments, which relied on knowing all of the moments of the distributions for which one was trying to prove convergence. In practice, however, finding higher order moments is often extremely difficult to do. The solution is the Stein-Chen method for proving convergence in distribution, introduced by Stein~\cite{Stein} for use with the normal distribution, and later modified by Chen~\cite{Chen} for use with the Poisson distribution. The power of the Stein-Chen method is that it only relies on knowing the first two moments of the distributions.

The version of the Stein-Chen method that we shall use is the following theorem of Barbour and Eagleson~\cite{BarEag}, which concerns a sum of Bernoulli random variables, each of which is dependent on only a small number of the other random variables.

\begin{theorem}\label{th:steinchen}
Let $X_1,\dots,X_n$ be Bernoulli random variables with $\P (X_i = 1) = p_i.$ Let $Y_n = \sum_{i=1}^{n}{X_i}$, and let $\lambda_n = \mathbb{E}Y_n = \sum_{i=1}^n p_i.$ For each $i\in [n],$ let $N_i \subset [n]$ be such that $X_i$ is independent of~$\{X_j:j\notin N_i\}$. For each $i,j \in [n],$ let $p_{ij} = \mathbb{E}X_i X_j$. Let $Z_n \sim \Po(\lambda_n)$. Then 
\[
\dTV(Y_n, Z_n) \leq \min\left\{1,\lambda_n^{-1}\right\} \Bigg( \sum_{i=1}^n \sum_{j\in N_i} p_i p_j + \sum_{i=1}^n \sum_{j\in N_i\setminus\{i\}} p_{ij} \Bigg).
\]
\end{theorem}

For further developments and applications of the Stein-Chen method, see, e.g.,~\cite{AGG,AGW,BCL,BHrate}, as well as~\cite{PoissonApprox} and the references therein.

\section{Extremal results}\label{se:extremal}

The aim of this section is to prove the combinatorial results needed in the proof of Theorem~\ref{th:crittime}. These results are all related to the event that a given site is uninfected at time~$t$.

We shall need a notion of distance between vertices. The appropriate distance for us is the $\ell_1$ distance, or graph distance, but, unfortunately, the $\ell_1$ norm is not a norm on $\Z^d$, nor (still less) is it a norm on $\T_n^d$. However, abusing notation slightly, we shall still write $\norm{x}$ for the length of the shortest path from the origin to~$x$, in $\Z^d$ or $\T_n^d$ as appropriate.

We define the $(d - 1)$-dimensional \emph{sphere} or \emph{layer} of radius~$t$ about a vertex~$x$ to be $S^{d-1}_t(x) := \{y \in \Z^d : \norm{y-x} = t\}$ and the $d$-dimensional \emph{ball} of radius~$t$ about $x$ to be $B^d_t(x) := \{y \in \Z^d : \norm{y-x} \leq t\}$. For short, we write $B_t := B_t^d(0)$ and $S_t := S_t^{d-1}(0)$. Recall that for each $i \in [d]$, $e_i$ denotes the $i$th standard basis vector of~$\R^d$. Given a vertex~$x$, we write $x_i$ for the $i$th coordinate of~$x$ relative to the standard basis vectors.  Thus, we have $x=\sum_{i=1}^d x_ie_i$.

We shall define a partial order~$\leq$ on $B_t$ by saying that $y \geq x$ if and only if for all~$i \in [d]$ such that $x_i \neq 0$, $y_i$ has the same sign as~$x_i$ and $\abs{y_i} \geq \abs{x_i}$. This gives us natural definitions of in- and out-neighbours: we say that $y$ is an \emph{in-neighbour} of~$x$ if $xy \in E(\T_n^d)$ and $y\leq x$, and similarly that $z$ is an \emph{out-neighbour} of~$x$ if $xz \in E(\T_n^d)$ and $z\geq x$.  If $y \geq x$, we shall sometimes say that $y$ is \emph{above} $x$.

Often, we shall need to talk about vertices that are uninfected at the last time that it could be important that they are uninfected. For a vertex~$x \in B_t$, this time is~$t-\norm{x}$; after this, the state of~$x$ cannot affect the state of the origin at time~$t$.  So, we say that a vertex~$x$ is \emph{protected} if it is uninfected at time~$t-\norm{x}$. We write $P(X)$ for the set of protected sites in a subset~$X$ of~$B_t$, and $P_k^+(x)$ for the set of protected sites~$y$ such that $y\geq x$ and $\norm{y - x} = k$ (it follows that we also have $\norm{y}=\norm{x}+k$). Note that an element of~$S_t$ is protected if and only if it is initially uninfected. Our original extremal question asked what one can say about the initial set~$A$ if the origin is protected.

\subsection{Minimal configurations}\label{se:exact}

Now we shall prove our main extremal result, the bound on the number of protected sites at a given distance from another protected site, which we may take to be the origin.  Given that the origin is protected, how might we go about proving that there are many protected sites? In two dimensions, it is relatively easy to check that the spheres~$S_k$ act independently, meaning that if some sphere~$S_k$ (with $k\leq t$) has too few uninfected sites, then the sites on that sphere alone will infect the origin by time~$t$ (in fact at time exactly~$k$). (In two dimensions, this minimum number is~4 for all~$k\geq 2$.) The spheres are likewise independent in $d \geq 2$ dimensions, and the proof of the result in $d$ dimensions makes key use of this independence of spheres. We show that the number of protected sites in $B_t$ is at least~$m_t$ by showing the stronger result that the number of protected sites in $S_k$ is at least a certain quantity for every~$k\leq t$.

How can we show that there must be many protected sites in $S_k$? Certainly, there must be at least~$d+1$ protected sites at distance~$1$ from the origin, otherwise the origin would not be protected. We would then like to say that because these $d+1$ sites are protected, there must be at least a certain number of protected sites at distance~$2$ from the origin. However, the sets of sites at distance~2 that protect these $d + 1$ sites at distance~1 could overlap. Thus, we would like an inductive argument which says that if a site~$x$ is protected then there must be many protected sites~$y$ at any given distance from $x$, all satisfying $y\geq x$. In other words, we would like a statement of the form, `if $x$ is protected, then $\abs{P_k^+(x)}\geq f_k(x)$', for some function~$f_k(x)$. What should $f_k(x)$ be? Clearly, its value should depend on the support of~$x$ (that is, the number of non-zero coordinates of~$x$): we are looking for protected vertices~$y$ such that $y \geq x$, and if $x$ has large support, then there are few such vertices, while if $x$ has small support, then there are many such vertices.

We take our cue from the column example,~\eqref{eq:column}, which we hope to prove is essentially the only minimal configuration.  Define
\begin{equation}\label{eq:minlayer}
\ell_t := \ell_{t,d} := \sum_{i = 0}^t \dbinom{d}{i}.
\end{equation}
Note that, by~\eqref{eq:minball}, we may write
\[
m_{t,d} = \sum_{r = 0}^t \ell_{r,d}.
\]
With this definition,~\eqref{eq:column} gives
\[
f_k(x) = \sum_{i=0}^k \dbinom{a}{i} = \ell_{k,a},
\]
where $a$ is the number of zero coordinates of~$x$. With that in mind, the following is our main lemma.  (We shall prove this result after further discussion.)

\begin{lemma}\label{le:keyweak}
Let $t\in\N$ and $d\geq 2$. Suppose that $x\in B_t$ is protected, let $k\leq t-\norm{x}$, and let $a$ be the cardinality of~$\{j \in [d] : x_j=0\}$. Then
\[
\abs{P_k^+(x)} \geq \sum_{i=0}^k \dbinom{a}{i}.
\]
\end{lemma}

In particular, this means that if the origin is protected, then
\[
\abs{P_t^+(0)} \geq \sum_{i=0}^t \dbinom{d}{i} = \ell_t.
\]

Before proving Lemma~\ref{le:keyweak}, let us look at an example. Suppose that $x=(t-k,0,\dots,0)$ is protected. Suppose also that we are in the fortunate position that $(t-i,0,\dots,0)$ is protected for each $i=0,1,\dots, k-1$. For a fixed $i$, we could then ask, given that $(t-i,0,\dots,0)$ is protected, how many protected sites~$y$ must there be in $S_t$ such that $y_1=t-i$? If we could get a good bound on this number then we would be in good shape: the condition $y_1=t-i$ ensures that these sets of protected sites are disjoint for different values of~$i$, so we could bound from below the number of protected sites in $S_t$ by summing the sizes of these sets. However, it is not clear that the minimum number is greater than~zero. In fact, if $i\geq d$ then the minimum number \emph{is}~zero. What about smaller values of~$i$? When $i=0$, the minimum is~$1$ (the site itself), and when $i=1$, the minimum is~$d-1$. Given the form of~$\ell_t$ in~\eqref{eq:minlayer}, it is tempting to view these numbers as $\binom{d-1}{0}$ and $\binom{d-1}{1}$ respectively, and to conjecture that the minimum for $i$ is~$\binom{d-1}{i}$.

Let us pause to see what this means. We are saying that if a site
is protected under $d$-neighbour bootstrap percolation in the $(d-1)$-dimensional space~$\Z^{d-1}$, then there must be at least~$\binom{d-1}{i}$ protected sites at distance~$i$, for each $i\leq d-1$. This assertion is strange, because we would not normally consider $r$-neighbour bootstrap percolation in a $d$-dimensional space for values of~$r$ greater than~$d$. However, it turns out that the assertion is true, and that it can be proved by a double counting argument. Thus, in the very special case in which $(t-i,0,\dots,0)$ is protected for each $0\leq i\leq k$, Lemma~\ref{le:keyweak} holds.

What happens if, for some~$i$, the site~$(t-i,0,\dots,0)$ is not protected? In that case, we must have two protected sites of the form~$(t-i-1,1,0,\dots,0)$ and~$(t-i-1,-1,0,\dots,0)$ (without loss of generality). If so, the sets of protected sites on $S_t$ that each of these sites generate (we presume by induction) will be disjoint: those generated by $(t-i-1,1,0,\dots,0)$ will have second coordinate at least~1, while those generated by $(t-i-1,-1,0,\dots,0)$ will have second coordinate at most~$-1$.  So we obtain two large, disjoint sets of protected sites on $S_t$.  Unfortunately, the sum of the sizes of these two sets is not quite large enough to give the bound in the lemma. However, we have not yet looked for any protected sites~$y$ with $y_2=0$. But this situation is now very similar to the previous case: we are asking how many sites~$y \in S_t$ with $y_2=0$ are needed to protect $(t-i-1,0,\dots,0)$. In other words, we are back to $d$-neighbour bootstrap percolation in a $(d-1)$-dimensional space, and the same double counting argument applies. We shall see that this gives us exactly the right number of additional protected sites to prove the lemma.

In general, the site~$x$ in question is not of the form~$(t-k,0,\dots,0)$, but the two-case argument above still applies. Either $x$ has a protected neighbour~$y$ such that $y\geq x$ and $y$ has the same number of zero coordinates as $x$, or $x$ has a pair of protected neighbours $x+e_j$ and~$x-e_j$ for some $j$ such that $x_j=0$. In both cases we obtain large sets of protected sites on $S_t$ by induction on $a + k$, where once again $a$ denotes the number of zero coordinates of~$x$ and $k \leq t - \norm{x}.$

This concludes the sketch of the proof of Lemma~\ref{le:keyweak}.  Using Lemma~\ref{le:keyweak}, we can determine the minimum number of uninfected vertices that are needed to protect the origin (Corollary~\ref{MinProtected}), as well as classify the extremal sets (Theorem~\ref{th:configs}).  However, when, in Section~\ref{se:stability}, we come to prove Theorem~\ref{th:stability} (the stability result), it will turn out that we require a slightly stronger statement than Lemma~\ref{le:keyweak}.  The proof of this result follows along the same lines as the argument described above.  So, rather than write out both proofs, we shall simply prove the stronger result.

Roughly speaking, in the proof of Theorem~\ref{th:stability}, we shall need to be able to be more specific about where we are looking for protected vertices. Given a protected site~$x$, we shall want not just to be able to say how many protected sites~$y$ there are above $x$, but also how many  protected sites~$y$ there are above $x$ with certain other restrictions on their coordinates. More specifically, we partition $[d]$ into three sets: $\pos$, the set of `positive directions'; $\NN$, the set of `negative directions'; and $\free$, the set of `free directions'. We then ask, given that $x$ is protected, how many protected vertices~$y$ must there be such that $\norm{y-x}=k$ and
\begin{enumerate*}
\item $y_i\geq x_i$ for $i\in\pos$;
\item $y_i\leq x_i$ for $i\in\NN$.
\end{enumerate*}
(For $i\in\free$ there is no restriction on $y_i$.) Lemma~\ref{le:keyweak} is the special case $\pos = \{i:x_i>0\}$, $\NN = \{i:x_i<0\}$, and $\free=[d]\setminus(\pos\cup\NN)$.

Once again, the proof of the stronger version of Lemma~\ref{le:keyweak} requires no new ideas: one should think of it as being what is really proved when one proves Lemma~\ref{le:keyweak}.

In preparation for the statement of the stronger result, let us formalize the definitions from the discussion above. We define a \emph{configuration} to be a function~$C : [d] \to \{-1, 0, 1\}$. If $C(i) = 0$, we say that $i$ is \emph{free} for $C$, and if $C(i) = 1$, we say that $i$ is \emph{positively constrained} for $C$. If $C(i) = -1$, we say that $i$ is \emph{negatively constrained} for $C$.  We define $\free(C)$, $\pos(C)$, and $\NN(C)$ to be, respectively, the set of free, positively constrained, and negatively constrained directions for~$C$. Thus, the sets $\free(C)$,~$\pos(C)$, and~$\NN(C)$ partition the set~$[d]$.

The configuration~$C$ determines where we can look for protected vertices on $S_t$. We say that $y$ is \emph{$C$-compatible} with~$x$ if $(y_i - x_i) C(i) \geq 0$ for all~$i \in [d]$. Let
\[
P_k^C(x) := \left\{y\in P(B_t) \, : \, \text{$y$ is $C$-compatible with $x$ and} \left\lVert y - x \right\rVert = k\right\}.
\]
For example, the set~$P_k^+(x)$ defined at the beginning of this section corresponds to~$P_k^C(x)$, where $C$ is the configuration defined by
\begin{equation*}
C(i) =
	\begin{cases}
	1, & x_i > 0, \\
	-1, & x_i < 0, \\
	0, & x_i = 0.
	\end{cases}
\end{equation*}
For an arbitrary confiuration $C$, we define the \emph{$C$-degree} of~$x$ to be the number of $C$-compatible neighbours of~$x$.  That is, $d^C(x) = \abs{P_1^C(x)}$.

We need a few more definitions relating to configurations. Let $\mathcal{C}_d$ denote the set of configurations. We define a partial order~$\leq$ on $\mathcal{C}_d$ such that $C \leq C'$ if for all~$i \in [d]$, either $C(i) = C'(i)$; or $C(i) \in \{-1, 1\}$ and $C'(i) = 0$. The unique maximal element of~$\mathcal{C}_d$ with respect to~$\leq$ is~$(0, \ldots, 0)$.  Note that if $C \leq C'$, then
\[
P_k^{C}(x) \subset P_k^{C'}(x).
\]

We say that two configurations $C$,~$C'$ are \emph{polar} if there exists $j \in [d]$ such that $C(j) C'(j) = -1$, and $C(k) = C'(k)$ for all~$k\neq j$. If $C$ and $C'$ are polar and $j \in [d]$ is the coordinate in which they differ, we say that $C$ is the \emph{$j$-polar opposite} of~$C'$, and conversely. We define the \emph{common parent} of polar configurations $C$ and~$C'$ to be the minimal configuration~$C''$ (with respect to~$\leq$) such that $C'' \geq C$ and $C'' \geq C'$. Because $C$ and $C'$ differ only in one coordinate, there is a unique configuration~$C''$ with the desired property.  Note that if $C$ is the $j$-polar opposite of~$C'$ and $C''$ is their common parent, then $C''(j)=0$.

If $C \in \mathcal{C}_d$ is such that $C(j) = 0$ for some $j\in [d]$, then we also define the converse relationship. The \emph{positive $j$-child} of~$C$ is the (unique) maximal $C' \leq C$ such that $C'(j) = 1$ and the \emph{negative $j$-child} is the $j$-polar opposite of~$C'$.

We define a related notion for neighbours of a vertex~$x$.  The vertex~$x$ has an \emph{opposing pair of protected neighbours} if both $x+e_j$ and $x-e_j$ are protected for some free coordinate~$j$.  Note that if $C$ and $C'$ are polar configurations and $C''$ is their common parent, then
\[
P^{C}_k(x + e_j) \cap P^{C'}_k(x - e_j) = \emptyset
\]
for all~$k \leq t - \norm{x}$.  Moreover,
\[
P^{C}_k(x + e_j), P^{C'}_k(x - e_j) \subset P^{C''}_{k+1}(x).
\]

Now we are ready to state and to prove Lemma~\ref{le:key}. In the applications that follow, we shall often set $k = t - \norm{x}$, $\pos(C) = \{i : x_i > 0\}$, $\NN(C) = \{i : x_i < 0\}$, and $\free(C) = \{i : x_i = 0\}$, and it may be helpful to think of them in this way in the proof of Lemma~\ref{le:key}.

\begin{lemma}\label{le:key}
Let $t \in \N$ and $d \geq 2$. Suppose that $x \in B_t$ is protected. Let $C \in \mathcal{C}_d$, let $a := \abs{\free(C)}$ denote the number of free coordinates of~$C$, and let $k \leq t - \norm{x}$ be a non-negative integer. Then
\begin{equation}\label{eq:ProtectedSum}
\left\lvert P_{k}^C(x) \right\rvert \geq \sum_{i=0}^k \binom{a}{i}.
\end{equation}
\end{lemma}

\begin{proof}
Without loss of generality, let $\pos(C) = [d] \setminus \free(C)$ and let $\NN(C) = \emptyset$. Let $s = d - a$. We may assume that $\pos(C) = [s]$ and that $\free(C) = \{s + 1, \ldots, d\}$. Let $x = (x_1, \ldots, x_d)$ and suppose that $x_i \geq 0$ for all~$i\in[d]$.

We shall prove the result by induction on $a + k$. If $a = k = 0$ then $x \in S_t$ and $P_0^C(x)=\{x\}$, so we have~\eqref{eq:ProtectedSum}. (In fact, we do not use $a=0$ here.)

Now suppose that the result holds for all values up to $a + k - 1$. We shall show that if $x$ has an opposing pair of $C$-compatible protected neighbours $x + e_i$ and $x - e_i$, then by induction there exist large, disjoint protected sets $P_{k-1}^{C'}(x + e_i)$ and~$P_{k-1}^{C''}(x - e_i)$ inside~$P_{k}^C(x)$, where $C'$ and $C''$ denote the positive and negative $i$-child of~$C$, respectively. Then we shall show that there exists an additional set of sites in $P_k^C(x)$, disjoint from both $P_{k-1}^{C'}(x + e_i)$ and $P_{k-1}^{C''}(x - e_i)$. If $x$ does not have an opposing pair of $C$-compatible protected neighbours, and $x'$ is a $C$-compatible protected neighbour of~$x$, then we shall show by induction that $P_{k-1}^{C}(x')$ has almost as many sites as we have claimed. We shall find the remaining protected vertices in $P_{k}^{C}(x)$ separately, disjoint from $P_{k-1}^{C}(x')$.

Thus, the remainder of the proof is split into two cases according to whether or not $x$ has an opposing pair of $C$-compatible protected neighbours. Once we have proved the result in both of these cases, the proof will be complete.

\emph{Case} 1: Suppose that $x$ has an opposing pair of $C$-compatible protected neighbours. Without loss of generality, suppose that $x^+ := x + e_{s + 1}$ and $x^- := x - e_{s + 1}$ are both protected. Let $C'$ and $C''$ denote the positive and negative $(s + 1)$-child of~$C$, respectively. Observe that
\[
\min_{y\in S_k(x)} \lVert x^{+} - y\rVert = \min_{y\in S_k(x)} \lVert x^{-} - y\rVert = k - 1,
\]
and that $C'$ and $C''$ each have $a-1$ free coordinates. Hence, by induction, $P_{k-1}^{C'}(x^+)$ and $P_{k-1}^{C''}(x^-)$ each contain at least~$\sum_{i=0}^{k-1} \binom{a-1}{i}$ vertices. By construction, we have
\[
P_{k-1}^{C''}(x^+) \cup P_{k-1}^{C'}(x^-) \subset P_k^C(x),
\]
and $P_{k-1}^{C''}(x^+) \cap P_{k-1}^{C'}(x^-) = \emptyset$. Thus,
\begin{align*}
\left\lvert P_{k-1}^{C'}(x^+) \cup P_{k-1}^{C''}(x^-) \right\rvert &\geq 2\left(\dbinom{a-1}{0} + \ldots + \dbinom{a-1}{k-1}\right)  \\
&=  \dbinom{a}{0} + \ldots + \dbinom{a}{k-1} + \dbinom{a-1}{k-1}.
\end{align*}

In order to prove~\eqref{eq:ProtectedSum}, we must show that $P_k^C(x)$ contains at least~$\binom{a}{k} - \binom{a-1}{k-1} = \binom{a-1}{k}$ additional sites. For each $j \geq 0$, let $P_j^0(x) = \{y \in P_j^C(x) : y_{s+1} = x_{s+1}\}$. By definition, $P_j^0(x)$ is disjoint from both $P_{k-1}^{C'}(x^+)$ and $P_{k-1}^{C''}(x^-)$. We shall show that $P_k^0(x)$ contains the required number of sites.

For each $j \geq 1$, let $G_j$ be the bipartite graph with classes~$P_{j-1}^0(x)$ and~$P_{j}^0(x)$, with two vertices adjacent if and only if they are adjacent in $B_t$. We shall obtain a lower bound on $\abs{P_j^0(x)}/\abs{P_{j-1}^0(x)}$ by double counting the edges in $G_j$.

Let $z \in P_{j}^0(x)$. We would like an upper bound on the degree of~$z$ in $G_j$. We obtained $z$ from $x$ by adding a total of~$j$ to some of the $d - 1$ coordinates of~$x$ other than the $(s+1)$th. To obtain a neighbour of~$z$ in $P_{j-1}^0(x)$, we have to subtract $1$ from one of the coordinates to which we have added at least~$1$. There are at most~$j$ such coordinates, so $z$ has at most~$j$ neighbours in $P_{j-1}^0(x)$.

Now let $y \in P_{j-1}^0(x)$. This time, we want a lower bound on the degree of~$y$ in $G_j$. We do this by bounding $d^C(y)$. We obtained $y$ from $x$ by adding a total of~$j - 1$ to some of the $d - 1$ coordinates of~$x$ other than the $(s+1)$th. For every $i$ such that $y_i = x_i$, $y$ has two $C$-compatible neighbours in direction~$i$, obtained by changing $y_i$ to~$x_i + 1$ or to~$x_i - 1$. If $y_i \neq x_i$ and $w$ is a $C$-compatible neighbour of~$y$ in direction~$i$, then $w_i - x_i$ must have the same sign as $y_i - x_i$, so there is only one such neighbour. This holds for each of the first $s$ coordinates, as well as for any coordinate~$i$ among the last $d - s - 1$ coordinates whose value we have already changed from $x_i$. It follows that $d^C(y)$ is minimized when we have already changed $j - 1$ of the last $d - s - 1$ coordinates. Hence,
\[
d^C(y) \geq 2(d - s - 1) + s - (j - 1) = 2d - s - j - 1.
\]
Since $y$ is protected, at most~$d - 1$ of its $C$-compatible neighbours are not protected. Hence, $y$ has at least
\[
2d - s - j - 1 - (d - 1) = d - s - j = a - j
\]
neighbours in $P_{j}^0(x)$.

By double counting the edges in $G_j$, first from $P_j^0(x)$ to~$P_{j-1}^0(x)$, and then from $P_{j-1}^0(x)$ to~$P_j^0(x)$, and using our bounds on the maximum and minimum degrees of vertices in these two classes, we obtain the inequalities
\[
j \left\lvert P_{j}^0(x) \right\rvert \geq \left\lvert E(G_j) \right\rvert \geq \left(a - j \right) \left\lvert P_{j-1}^0(x) \right\rvert.
\]
Thus,
\[
\left\lvert P_{j}^0(x) \right\rvert \geq \dfrac{a - j}{j} \left\lvert P_{j-1}^0(x) \right\rvert.
\]
Noting that $\abs{P_0^0(x)} = 1$, it follows by induction on $j$ that $\abs{P_j^0(x)} \geq \binom{a-1}{j}$. Taking $j=k$ proves~\eqref{eq:ProtectedSum} in the case where $x$ has an opposing pair of $C$-compatible protected neighbours.

\emph{Case} 2: Suppose that $x$ has no opposing pair of $C$-compatible protected neighbours. Observe that the total number of $C$-compatible neighbours of~$x$ is~$d+a$. Since $x$ is protected, at most~$d-1$ of its neighbours are not protected, so at least~$a+1$ of its $C$-compatible neighbours are protected. Since $x$ has no opposing pair of $C$-compatible protected neighbours, at most~$a = d - s$ of these are of the form~$x\pm e_i$ with $s+1\leq i\leq d$, so $x$ must have at least~one protected neighbour of the form~$x + e_i$ with $1 \leq i \leq s$. Without loss of generality, let $x' := (x_1 + 1, x_2, \ldots, x_d)$ be protected. We have
\[
\min_{y\in S_k(x)} \lVert x' - y\rVert = k - 1,
\]
and $x'$ still has $a$ free coordinates. Hence by induction, $P_{k-1}^C(x') \subset P_k^C(x)$ is such that
\[
\left\lvert P_{k-1}^C(x') \right\rvert \geq \sum_{i=0}^{k-1} \dbinom{a}{i}.
\]
In order to prove~\eqref{eq:ProtectedSum}, we need to find an additional $\binom{a}{k}$ sites in $P_k^C(x)$ disjoint from the sites in $P_{k-1}^C(x')$.

All of the elements of~$P_{k-1}^C(x')$ have first coordinate at least~$x_1 + 1$. For each $j \geq 0$, let $Q_j(x) = \{y \in P_j^C(x) : y_1 = x_1\}$. By definition, $Q_j(x) \cap P_{j-1}^C(x') = \emptyset$ for all~$j\geq 1$. We shall show that $Q_k(x)$ contains the required number of sites.

For each $j\geq 1$, let $H_j$ be the bipartite graph with classes~$Q_{j-1}(x)$ and~$Q_j(x)$ in which two vertices are adjacent if and only if they are adjacent in $B_t$. As in Case 1, we bound $\abs{Q_j(x)}/\abs{Q_{j-1}(x)}$ by double counting edges in $H_j$.

As before, any element of~$Q_j(x)$ has at most~$j$ neighbours in $Q_{j-1}(x)$. Let $y \in Q_{j-1}(x)$. Then $y$ has two $C$-compatible neighbours in each of the at most~$a$ coordinates~$i$ for which $y_i = x_i$, but only one $C$-compatible neighbour in each coordinate~$i$ for which $y_i \neq x_i$. Again, the degree~$d^C(y)$ is minimized when we have obtained $y$ from $x$ by changing the value of~$j-1$ of the last $a$ coordinates. Hence,
\[
d^C(y) \geq 2(d - s) + (s - 1) - (j - 1) = 2d - s - j.
\]
At most~$d-1$ of the $C$-compatible neighbours of~$y$ are not protected, so $y$ has at least
\[
2d - s - j - (d - 1) = d - s - j + 1 = a - j + 1
\]
neighbours in $Q_j(x)$. Therefore,
\[
j \lvert Q_{j}(x) \rvert \geq \lvert E(H_j) \rvert \geq (a - j + 1) \lvert Q_{j-1}(x) \rvert,
\]
and thus
\[
\lvert Q_{j}(x) \rvert \geq \dfrac{a - j + 1}{j} \lvert Q_{j-1}(x) \rvert.
\]
Because $\abs{Q_0(x)} = 1$, it follows by induction on $j$ that $\abs{Q_j} \geq \binom{a}{j}$ for all~$j\geq 0$, as required. This completes the case where $x$ does not have an opposing pair of $C$-compatible neighbours, and hence also the proof of the lemma.
\end{proof}

It follows immediately from Lemma~\ref{le:key} that the minimum size of a subset of~$S_t$ that protects the origin is~$\ell_t$, and hence that the minimum size of a subset of~$B_t$ that protects the origin is~$m_t$.

\begin{corollary}\label{MinProtected}
Let $t \in \N$ and $d \geq 2$. Suppose that the origin is protected. Then
\[
\lvert P(S_t) \rvert \geq \ell_t.
\]
In particular, if $t \geq d$, then $S_t$ contains at least~$2^d$ protected vertices. Moreover,
\[
\lvert P(B_t) \rvert \geq m_t.
\]
\end{corollary}

\begin{proof}
Apply Lemma~\ref{le:key} to the origin with $C = (0, \ldots, 0)$ and $k = 1$,~$2$, $\ldots\,$,~$t$.
\end{proof}

Now we classify the extremal sets of protected vertices in $B_t$. Recall from the introduction that the motivation for Lemma~\ref{le:key} came from the conjecture that the extremal configurations should all be (what we have so far called) columns, as in~\eqref{eq:column}. In the next theorem, we prove that this is essentially correct. More specifically, we prove that the only extremal sets are either columns or sets that are almost columns except for the top and bottom sites.

Formally, we call $P(B_t)$ \emph{canonical} if there exists $j \in [d]$ and an \emph{orientation} $\varepsilon_i \in \{-1, 1\}$ for each $i \in [d]\setminus\{j\}$ such that
\begin{equation}\label{eq:canon}
P(B_t) = \{ x \in B_t \, : \, x_i \in \{0, \varepsilon_i\} \text{ for all } i\neq j \}.
\end{equation}
 Given $j$ and the $\varepsilon_i$, let
\[
V_j^+(t) = \{te_j\} \cup \bigl\{(t - 1)e_j - \varepsilon_i e_i \, : \, i \in [d] \setminus \{j\} \bigr\}
\]
and let
\[
V_j^-(t) = \{-te_j\} \cup \bigl\{(-t + 1)e_j - \varepsilon_i e_i \, : \, i \in [d] \setminus \{j\} \bigr\}.
\]
We call $P(B_t)$ \emph{semi-canonical} if there exist $v^+ \in V_j^+(t)$ and $v^- \in V_j^-(t)$ such that
\begin{equation}\label{eq:semicanon}
P(B_t) = \big( \{ x \in B_t \, : \, x_i \in \{0, \varepsilon_i\} \text{ for all } i\neq j \} \setminus \{te_j,-te_j\} \big) \cup \{v^+, v^-\}.
\end{equation}
Note that canonical sets are semi-canonical. We call the vertices $v^+$ and~$v^-$ the \emph{extreme points} of~$P(B_t)$. The direction~$j$ is the \emph{direction of alignment} of~$P(B_t)$, and $P(B_t)$ is said to be \emph{$j$-aligned}.

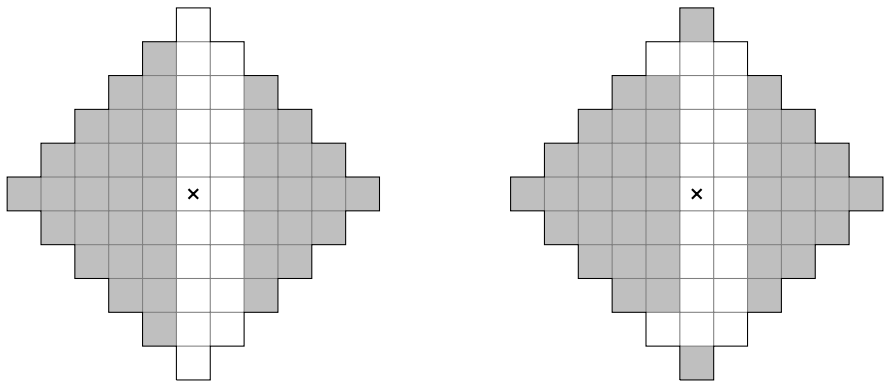
\begin{figure}[ht]
 \begin{minipage}{.45\textwidth}
  \centering
  \begin{tikzpicture}[scale=0.45,>=latex]
   \path [fill=gray!50] (4,1) rectangle (7,10) (3,2) rectangle (8,9) (2,3) rectangle (9,8) (1,4) rectangle (10,7) (0,5) rectangle (11,6);
   \path [fill=white] (5,0) rectangle (6,11) (6,1) rectangle (7,10);
   \draw [gray] (5,1) rectangle (6,10) (4,2) rectangle (7,9) (3,3) rectangle (8,8) (2,4) rectangle (9,7) (1,5) rectangle (10,6);
   \draw [black] (5,0) -- (6,0) -- (6,1) -- (7,1) -- (7,2) -- (8,2) -- (8,3) -- (9,3) -- (9,4) -- (10,4) -- (10,5) -- (11,5) -- (11,6) -- (10,6) -- (10,7) -- (9,7) -- (9,8) -- (8,8) -- (8,9) -- (7,9) -- (7,10) -- (6,10) -- (6,11) -- (5,11) -- (5,10) -- (4,10) -- (4,9) -- (3,9) -- (3,8) -- (2,8) -- (2,7) -- (1,7) -- (1,6) -- (0,6) -- (0,5) -- (1,5) -- (1,4) -- (2,4) -- (2,3) -- (3,3) -- (3,2) -- (4,2) -- (4,1) -- (5,1) -- (5,0);
   \node [cross out,draw,thick,inner sep=0,minimum size=0.1cm] at (5.5,5.5) {};
  \end{tikzpicture}
 \end{minipage}
 \begin{minipage}{.45\textwidth}
  \centering
  \begin{tikzpicture}[scale=0.45,>=latex]
   \path [fill=gray!50] (3,2) rectangle (8,9) (2,3) rectangle (9,8) (1,4) rectangle (10,7) (0,5) rectangle (11,6) (5,0) rectangle (6,1) (5,10) rectangle (6,11);
   \path [fill=white] (5,1) rectangle (7,10);
   \path [fill=gray!50] (5,0) rectangle (6,1) (5,10) rectangle (6,11);
   \draw [gray] (5,1) rectangle (6,10) (4,2) rectangle (7,9) (3,3) rectangle (8,8) (2,4) rectangle (9,7) (1,5) rectangle (10,6);
   \draw [black] (5,0) -- (6,0) -- (6,1) -- (7,1) -- (7,2) -- (8,2) -- (8,3) -- (9,3) -- (9,4) -- (10,4) -- (10,5) -- (11,5) -- (11,6) -- (10,6) -- (10,7) -- (9,7) -- (9,8) -- (8,8) -- (8,9) -- (7,9) -- (7,10) -- (6,10) -- (6,11) -- (5,11) -- (5,10) -- (4,10) -- (4,9) -- (3,9) -- (3,8) -- (2,8) -- (2,7) -- (1,7) -- (1,6) -- (0,6) -- (0,5) -- (1,5) -- (1,4) -- (2,4) -- (2,3) -- (3,3) -- (3,2) -- (4,2) -- (4,1) -- (5,1) -- (5,0);
   \node [cross out,draw,thick,inner sep=0,minimum size=0.1cm] at (5.5,5.5) {};
  \end{tikzpicture}
 \end{minipage}
 \caption{On the left, a canonical set of protected sites, and on the right, a semi-canonical set of protected sites.}
\end{figure}

We are ready to state the main theorem of this section. We say that the set of uninfected sites in $S_t$ is \emph{minimal} and that $S_t$ is a \emph{minimal layer} if $\abs{P(S_t)} = \ell_t$. Similarly, we say that $B_t$ is \emph{minimal} if $\abs{P(B_t)}=m_t$.

\begin{theorem}\label{th:configs}
Let $t$, $d \geq 2$. Suppose that the origin is protected and that $B_t$ is minimal. Then $P(B_t)$ is semi-canonical.
\end{theorem}

We write $d_X(x)$ for the degree of~$x$ in the set~$X$, so
\[
d_X(x) = \bigl\lvert\{y\in X \, : \, \lVert x-y\rVert = 1\}\bigr\rvert.
\]
We also set $P_r = P(S_r)$.

\begin{proof}
We shall show that if $P(B_k)$ is semi-canonical and $S_{k+1}$ is minimal, then the only sets of sites in $S_{k+1}$ with enough neighbours in $P_k$ are the sets which make $P_{k+1}$ semi-canonical, and moreover that if $P(B_k)$ is not canonical then there are no sets of sites in $S_{k+1}$ with enough neighbours in $P(B_k)$. We shall repeatedly use the following equation for double counting edges.
\begin{equation}\label{eq:EqualDegrees}
\sum_{v \in P_k} d_{P_{k+1}}(v) = \sum_{w \in P_{k+1}} d_{P_k}(w).
\end{equation}

The sphere~$S_1$ is minimal, so by Corollary~\ref{MinProtected} $P_1$ consists of~$d+1$ sites. Suppose that $P(B_1)$ is not canonical. Then there must exist $i$, $j \in [d]$ such that $e_i$, $-e_i$, $e_j$, and $-e_j$ all belong to~$P_1$. A site in $S_2$ has degree~$2$ in $P_1$ if and only if it is of the form~$x + y$ for some $x,y \in P_1$ such that $x + y \neq 0$; otherwise, it has degree~$1$. Thus, if $m$ is the number of sites in $S_2$ with degree~$2$ in $P_1$, then $m \leq \binom{d+1}{2} - 2$. Let $Q$ be any set of~$\ell_2$ sites in $S_2$. Then
\begin{equation}\label{eq:S2in}
\sum_{v \in Q} d_{P_1}(v) \leq \ell_2 + m \leq d^2+d-1.
\end{equation}
From below, note that protected sites in $S_1$ have at least~$d$ protected out-neighbours. Therefore
\begin{equation}\label{eq:S1out}
\sum_{w \in P_1} d_{Q}(w) \geq d(d+1),
\end{equation}
which by \eqref{eq:EqualDegrees} and~\eqref{eq:S2in} is a contradiction. Therefore $P(B_1)$ is canonical.

The choice of~$P(B_1)$ determines the values of~$j$ and the $\varepsilon_i$. Throughout the rest of the proof, without loss of generality, let $j = 1$ and $\varepsilon_2 = \ldots = \varepsilon_d = 1$. Then $P_1 = \{e_1, -e_1, e_2, e_3, \ldots, e_d\}$.

We use induction on $k$. First, suppose that $P(B_k)$ is canonical. We shall show that $P(B_{k+1})$ is semi-canonical by double counting edges between consecutive spheres. Later we show that if $P(B_k)$ is semi-canonical but not canonical, then it is not possible to increase the number of minimal layers, and there is a contradiction.

Let $R_{k+1}$ denote the set of sites in $S_{k+1}$ with at least~two neighbours in $P_k$. As with $k=1$, our aim is to show that $R_{k+1} \subset P_{k+1}$.  First, we show that for all~$k$,
\[
\abs{R_{k+1}} = \ell_{k+1} - 2.
\]
There are two cases to consider.  First, suppose that $k \geq d - 1$.  In this case, no vertex~$y \in R_{k+1}$ is such that $y_1 = 0$.  If $y \in R_{k+1}$ has $j$ non-zero coordinates among its last $d-1$ coordinates, then $d_{P_k}(y) = j+1$. Since all sites in $R_{k+1}$ have degree at least~$2$ in $P_k$, we must have $j\geq 1$. For each~$j$, there are $2\binom{d-1}{j}$ such sites in $R_{k+1}$ (the factor of~$2$ comes from the two choices for the first coordinate). Hence,
\[
\abs{R_{k+1}} = 2\sum_{j=1}^{d-1} \dbinom{d-1}{j} = 2^d - 2,
\]
as required.

Second, suppose that $k \leq d - 2$.  In this case, there do exist vertices~$y \in R_{k+1}$ with $y_1 = 0$.  As in the previous case, every vertex in $R_{k+1}$ must have at least~one of its last $d - 1$ coordinates not equal to~0.  Once again, for each~$j \in [k]$, there are $2\binom{d-1}{j}$ sites in $R_{k+1}$ with $j$ non-zero coordinates among their last $d-1$ coordinates.  If $k+1$ of the last $d-1$ coordinates of~$y \in P_k$ do not equal $0$, then we must have $y_1 = 0$.  Thus, we have
\[
\abs{R_{k+1}} = \dbinom{d-1}{k+1} + 2\sum_{j=1}^k \dbinom{d-1}{j} = \ell_{k+1} - 2,
\]
as claimed.

It follows from the definition of~$R_{k+1}$ that
\begin{equation}\label{eq:equalityR}
\sum_{x \in P_{k+1}} d_{P_k}(x) \leq \sum_{x \in R_{k+1}} d_{P_k}(x) + 2.
\end{equation}
Because $S_{k+1}$ is minimal, equality holds in~\eqref{eq:equalityR} only if $R_{k+1} \subset P_{k+1}$.  Observe that by induction, $R_{k+1}$ is precisely the canonical set for $S_{k+1}$, except for the extreme points.  Hence,
\[
\sum_{x \in R_{k+1}} d_{P_k}(x) = \sum_{y \in P_k} d_{P_{k+1}}(y) - 2.
\]
This means that equality holds in~\eqref{eq:equalityR}, which implies that $R_{k+1} \subset P_{k+1}$.

Every vertex~$x \in P_k$ has enough neighbours in $R_{k+1}$ except for $ke_j$ and $-ke_j$, which each have one neighbour in $P_{k-1}$ but only $d - 1$ neighbours in $R_{k+1}$. Therefore we must have $v^+ \in V^+(k + 1)$ and $v^- \in V^-(k + 1)$ in $P_{k+1}$, too. So
\[
P_{k+1} = R_{k+1} \cup \{v^+, v^-\}
\]
is semi-canonical. This completes the proof of the theorem in the case where $P_k$ is canonical.

Now we show that $P(B_k)$ being non-canonical leads to a contradiction. Without loss of generality, suppose that $x = (t-1,-1,0,\dots,0)$ belongs to~$P_k$. Then $d_{P_{k-1}}(x) = 1$, which means that $x$ needs $d$ protected neighbours in $S_{k+1}$. However, no element of~$R_{k+1}$ is adjacent to~$x$. This means that in order for $P_{k+1}$ to satisfy~\eqref{eq:EqualDegrees} we need to add $d \geq 2$ sites to~$R_{k+1}$ to protect $x$, as well as at least~one site to protect the other extreme point of~$P(B_k)$, contradicting the assumption that $B_{k+1}$ is minimal. This completes the proof.
\end{proof}

\begin{corollary}\label{NumExtConfigs}
Let $t$, $d \geq 2$. Then there are exactly~$d^3 2^{d-1}$ minimal configurations of protected sites in $B_t$. \qed
\end{corollary}

\subsection{Near-minimal configurations}\label{se:stability}

Suppose that the number of protected sites in $B_t$ is not~$m_t$, but~$m_t+k$, where $k$ is small. What can we say about the positions of these sites? We would like to be able to show that they are not too far from being a canonical set together with an additional $k$ arbitrarily placed sites. Such a result would be interesting in its own right, but it also turns out that it is needed in the proof of the main theorem, Theorem~\ref{th:crittime}, to establish tighter bounds on the mean and variance of the number of uninfected sites at time~$t$. In particular, it will be important in the proof of Theorem~\ref{th:crittime} that the number of near-minimal configurations just described is~$O(t^{ck})$, where $c$ only depends on~$d$, and not~$O(t^{ctk})$, which is the trivial bound.

The stability result that we shall prove is the following.

\begin{theorem}\label{th:stability}
Let $t \geq 4d + 1$. Let $r_1$, $r_2$ be such that $r_1 \geq d$, $r_2 - r_1 \geq 3d + 1$, and~$r_2\leq t$. Suppose that the origin is protected and that $S_r$ is minimal for all~$r$ such that $r_1\leq r\leq r_2$. Then $P(B_{r_2})$ is semi-canonical. 
\end{theorem}

At the beginning of the proof of Theorem~\ref{th:stability}, we make crucial use of the following lemma, which says that a wide band of minimal layers can only have two connected components that meet the middle layer of the band. The rather strong condition $r_2-r_1\geq 3d+1$ in the statement of Theorem~\ref{th:stability} comes from this lemma.

\begin{lemma}\label{le:components}
Let $d \geq 2$, $r\geq d$, $s \geq d/2$, and $t \geq r+2d+2s$. Suppose that the origin is protected and that layers $S_r,\dots,S_{r+2d+2s}$ are all minimal. Then the set of uninfected sites in $B_{r+2d+2s} \setminus B_{r-1}$ contains at most~two connected components that meet $S_{r+d+s}$.
\end{lemma}

\begin{proof}
Let $x \in S_{r+d+s}$ be an uninfected site. Then $x$ is protected, because $S_{r+d+s}$ is minimal. Since $t-(r+d+s)\geq d+s$, we can apply Lemma~\ref{le:key} to~$x$ with $C=(0,\dots,0)$ and values of~$k$ up to~$d+s$. This gives that the component of uninfected sites containing $x$ inside $B_{r+2d+2s} \setminus B_{r-1}$ has size at least~$m_{d+s}$. Using the identity 
\[
\sum_{r = 0}^{d - 1} \sum_{j = 0}^r \dbinom{d}{j} = d 2^{d - 1}
\]
and the definition of~$m_t$ from~\eqref{eq:minball}, it follows that the size of the component containing $x$ is at least
\[
m_{d+s} = (s+1)2^d + d2^{d-1}.
\]
Do the same for every component of the set of uninfected sites that meets $S_{r+d+s}$. Let $K$ denote the number of components of uninfected sites in $B_{r+2d+2s} \setminus B_{r-1}$ that meet $S_{r+d+s}$. Let $N$ denote the total number of uninfected sites in these components. Then
\begin{equation}\label{eq:healthyLB}
K \big((s+1)2^d + d2^{d-1}\big) \leq N.
\end{equation}
Because these layers are minimal and the origin is protected, it follows from Corollary~\ref{MinProtected} that
\begin{equation}\label{eq:healthyUB}
N \leq (2d+2s+1) 2^d.
\end{equation}
Combining bounds \eqref{eq:healthyLB} and~\eqref{eq:healthyUB} we have
\[
K \leq \dfrac{2d+2s+1}{s+1+d/2} < 3.
\]
But $K$ is an integer, so we must have $K \leq 2$, as claimed.
\end{proof}

We need one more technical lemma before we prove Theorem~\ref{th:stability}.  Given a configuration~$C$, we say that a protected vertex~$x$ is \emph{$C$-supported} if for all~$I \subset [d]$ the vertex~$x - \sum_{i \in I} C(i)e_i$ is protected.

\begin{remark}\label{re:Csupport}
The property of being $C$-supported is `monotone' in the following sense.  Let $C \in \mathcal{C}_d$ and let $x$ be a $C$-supported vertex.  By definition, for all~$i$, $x - C(i)e_i$ is protected.  Let $C^0$ be the configuration obtained from $C$ by changing $C(i)$ to~0; observe that $C \leq C^0$, where $\leq$ is the partial order defined on $\configs$.  Then $x - C(i)e_i$ is $C^0$-supported.
\end{remark}

\begin{lemma}\label{le:minCcompat}
Let $r \geq d \geq 2$ and suppose that $S_r$ is minimal. Let $C \in \mathcal{C}_d$ and let $i=\abs{\pos(C)} + \abs{\NN(C)}$. Let $x$ be a $C$-supported vertex satisfying $r - \norm{x} \geq d - i$.  Then
\[
\bigl\lvert P_{r - \norm{x}}^C(x) \bigr\rvert = 2^{d-i}.
\]
\end{lemma}

\begin{proof}
For $i = 0$, the argument is very similar to the proof of Corollary~\ref{MinProtected}: simply apply Lemma~\ref{le:key} to~$x$ for $k = 0$,~$1$, $\ldots\,$,~$\norm{x}$. Proceeding by induction on $i$, let $C$ be a configuration with $\abs{\pos(C)} + \abs{\NN(C)} = i$.  Without loss of generality, suppose that $\pos(C) \neq \emptyset$ and that $j \in \pos(C)$.  Observe that, by hypothesis, the vertex~$x - e_j$ is protected.  Let $C'$ be the $j$-polar opposite of~$C$ and let $C''$ be their common parent.  Observe that $\abs{\pos(C'')} + \abs{\NN(C'')} = i - 1$.  We now consider the sets of protected vertices in $S_r$ generated by the vertices $x$ and~$x - e_j$.  Observe that
\[
P^C_{r - \norm{x}}(x), P^{C'}_{r - \norm{x} + 1}(x - e_j) \subset P^{C''}_{r - \norm{x} + 1}(x - e_j),
\]
as well as that
\[P^C_{r - \norm{x}}(x) \cap P^{C'}_{r - \norm{x} + 1}(x - e_j) = \emptyset.
\]
By Lemma~\ref{le:key} and the fact that $r - \norm{x} \geq d - i$, we have $\abs{P_{r-\norm{x}}^{C}(x)}$, $\abs{P_{r-\norm{x} + 1}^{C'}(x - e_j)} \geq 2^{d-i}$.  Also, by Remark~\ref{re:Csupport}, $x - e_j$ is $C''$-supported.  So, by the induction hypothesis, we have
\[
\bigl\lvert P_{r-\norm{x}}^{C}(x) \bigr\rvert + \bigl\lvert P^{C'}_{r - \norm{x} + 1}(x - e_j) \bigr\rvert = \bigl\lvert P^{C''}_{r - \norm{x} + 1}(x - e_j) \bigr\rvert = 2^{d-i+1},
\]
and the result follows.
\end{proof}

Recall that a vertex~$x$ has an opposing pair of protected neighbours if, for some~$j$, both $x + e_j$ and $x - e_j$ are protected. We say that $x$ is \emph{$j$-oriented} if $j$ is the unique coordinate for which both $x + e_j$ and $x - e_j$ are protected.

Here is a sketch of the proof of Theorem~\ref{th:stability}. As usual, we first consider the origin, which must have at least~$d+1$ protected neighbours. Hence it has two opposing protected neighbours, which without loss of generality are $e_1$ and $-e_1$. Next we show that $se_1$ and $-se_1$ are protected for all~$s\leq r_2-d$. We do this inductively: for each~$s$ and any $i>1$, we show that if $se_1$ is protected then it can never be the case that both $se_1+e_i$ and $se_1-e_i$ are protected. It then follows, because $se_1$ has at least~$d$ protected out-neighbours, that $(s+1)e_1$ is protected. How do we show that it is not the case that both $se_1+e_i$ and $se_1-e_i$ are protected? Well, if they are, then for $r_1 \leq r \leq r_2$, Lemma~\ref{le:key} gives $2^{d-2}$ protected sites~$y$ in $S_r$ with $y_1\geq 1$ and $y_i\geq 1$, and a further $2^{d-2}$ protected sites~$z$ in $S_r$ with $z_1\geq 1$ and $z_i\leq -1$. The same lemma applied to~$-e_1$ also gives $2^{d-1}$ protected sites~$w$ in $S_r$ with $w_1\leq -1$. These sets are disjoint, and $S_r$ is minimal, so we have found all of the protected sites in $S_r$. This holds for all~$r$ in the range $r_1\leq r\leq r_2$. Note that this means that there are at least~three components of protected sites that meet $S_{r_1+3d/2}$ in this band of minimal layers, which contradicts the components lemma, Lemma~\ref{le:components}. This is the only point in the proof where we use this lemma.

Knowing that $se_1$ is protected for all~$s \leq r_2 - d$ allows to us to show that we cannot have both $e_2$ and $-e_2$ protected. Let $r_1 \leq r \leq r_2$ and suppose that $e_2$ and $-e_2$ are both protected. Applying Lemma~\ref{le:key} to each of them in turn gives a total of~$2^d$ protected sites in $S_r$, all with second coordinate non-zero. But $re_1$ is also in that layer, and it is also protected, which is a contradiction. This idea of finding all (or as we shall see in a moment, a subset of) the protected sites in $S_r$ and showing that this leads to a contradiction by finding another protected site somewhere else in $S_r$ is one that we shall use repeatedly throughout the proof.

At this stage we know without loss of generality that
\[
P(S_1)=\{e_1,-e_1,e_2,\dots,e_d\},
\]
and also that $se_1$ and $-se_1$ are protected for all~$s\leq r_2-d$. As one would expect, from here we build the column inductively, in this case by induction on the number~$k$ of non-zero coordinates of the site. At each stage of the induction we show three things. First, that all sites with first coordinate zero and other coordinates consisting of~$k$ ones and $d-k-1$ zeros are protected. Second, that all sites above and below these sites are protected. By this we mean that sites obtained by adding or subtracting $se_1$ from one of these sites are also protected. Third, that all of these sites are $1$-oriented (recall that this means that they do not have opposing protected neighbours, except in the first coordinate). All three of these assertions are proved using variations of the same argument. In each case, if the assertion fails then there are always two protected sites, $x$~and~$y$, that differ in one of their coordinates (other than the first) by exactly~$2$. For example, suppose that $x_2=2$ and $y_2=0$. A combination of Lemma~\ref{le:key} and Lemma~\ref{le:minCcompat} tells us exactly how many protected sites there are in $S_r$ which are $C$-compatible with a certain $C$-supported site~$z$, where $z$ is such that $z \leq x$ and $z \leq y$ and $C$ is a suitable configuration, and it also tells us that none of them have (in this example) second coordinate equal to~$1$. The contradiction comes from knowing that in fact there is a protected site in $S_r$ which is $C$-compatible with the origin and has second coordinate equal to~$1$.  Once we have finished building the column, the final step, showing that $P(B_{r_2})$ is semi-canonical, follows easily.

\begin{proof}[Proof of Theorem~\ref{th:stability}]
Once again, we write $P_k$ for $P(S_k)$.  The origin must have at least~$d+1$ protected neighbours, so it must have an opposing pair of protected neighbours. Without loss of generality, suppose that both $e_1=(1, 0, \ldots, 0)$ and $-e_1=(-1, 0, \ldots, 0)$ are protected. We shall show that $se_1$ and $-se_1$ are protected for all~$s\leq r_2-d$. To do this, first we show that neither $e_1$ nor $-e_1$ has an opposing pair of protected neighbours in any direction except~$1$. Suppose for some~$i \neq 1$ that $e_1 + e_i$ and $e_1 - e_i$ are both protected. Define a configuration~$C \in \mathcal{C}_d$ by $C(1)=C(i)=1$ and $C(k)=0$ otherwise. Let $C'$ be the $i$-polar opposite of~$C$, so that $C'(1)=1$, $C'(i)=-1$, and $C'(k)=0$ otherwise. Fix $r$, $r_1 \leq r \leq r_2$.  By applying Lemma~\ref{le:key} to~$e_1 + e_i$ with $C$, we find a set~$Q_r^+$ of~$2^{d-2}$ protected sites~$y$ in $S_r$ with $y_1\geq 1$ and $y_i\geq 1$. Similarly, applying Lemma~\ref{le:key} to~$e_1 - e_i$ with $C'$, we find a set~$Q_r^-$ of~$2^{d-2}$ protected sites~$z$ in $S_r$ with $z_1\geq 1$ and $z_i\leq -1$. In addition, applying Lemma~\ref{le:key} to~$-e_1$ with configuration~$C''$ defined by $C''(1)=-1$ and $C''(k)=0$ otherwise, we find another set~$T_r$ of~$2^{d-1}$ protected sites~$w$ in $S_r$ with $w_1 \leq -1$, for a total of~$2^d$ protected sites in $S_r$. Let
\[
Q^+=\bigcup_{r=r_1}^{r_2} Q_r^+, \quad Q^-=\bigcup_{r=r_1}^{r_2} Q_r^-, \quad \text{and} \quad T=\bigcup_{r=r_1}^{r_2} T_r.
\]
The situation is shown in Figure~\ref{fi:3comps}.

\begin{figure}[ht]
  \centering
  \begin{tikzpicture}[>=latex]
    \path [name path=C2,fill=gray!50] (0,0) circle (3);
    \path [name path=C1,fill=white] (0,0) circle (2);
    \path [name path=L1] (0,-0.4) -- (-3,-0.4);
    \path [name path=R1] (0,-0.4) -- (3,-0.4);
    \path [name path=L2] (-0.4,0.4) -- (-3,0.4);
    \path [name path=R2] (0.4,0.4) -- (3,0.4);
    \path [name intersections={of=C2 and L1,by=P1}];
    \path [name intersections={of=C2 and R1,by=Q1}];
    \path [name intersections={of=C2 and L2,by=P2}];
    \path [name intersections={of=C2 and R2,by=Q2}];
    \path [name path=U1] (-0.4,0.4) -- (-0.4,3);
    \path [name path=U2] (0.4,0.4) -- (0.4,3);
    \path [name intersections={of=C2 and U1,by=R1}];
    \path [name intersections={of=C2 and U2,by=R2}];
    \path [fill=white] (-0.4,0.4) -- (-0.4,3) -- (0.4,3) -- (0.4,0.4) -- (3,0.4) -- (3,-0.4) -- (-3,-0.4) -- (-3,0.4) -- (-0.4,0.4);
    \draw (0,0) circle (3);
    \draw (0,0) circle (2);
    \draw (P1) -- (Q1) (P2) -- (-0.4,0.4) (0.4,0.4) -- (Q2);
    \draw (-0.4,0.4) -- (R1) (0.4,0.4) -- (R2);
    \draw [dashed,->] (-3,0) -- (3.8,0) node [right] {$e_i$};
    \draw [dashed,->] (0,0) -- (0,3.8) node [above] {$e_1$};
    \node [circle,fill,inner sep=0,minimum size=0.1cm] at (0,0) {};
    \draw (0.4,0.4) node [circle,fill,inner sep=0,minimum size=0.1cm] {} -- (0.8,1) -- (3.5,1) node [right] {$e_1+e_i$};
    \draw (-0.4,0.4) node [circle,fill,inner sep=0,minimum size=0.1cm] {} -- (-0.8,1) -- (-3.5,1) node [left] {$e_1-e_i$};
    \draw (0,-0.4) node [circle,fill,inner sep=0,minimum size=0.1cm] {} -- (0.4,-1) -- (3.5,-1) node [right] {$-e_1$};
    \node at (50:2.5) {$Q^+$};
    \node at (130:2.5) {$Q^-$};
    \node at (-90:2.5) {$T$};
    \path [name path=A1] (0,2.5) -- (-3,2.5);
    \path [name intersections={of=C2 and A1,by=S1}];
    \path [name path=A2] (0,-1.5) -- (-3,-1.5);
    \path [name intersections={of=C1 and A2,by=S2}];
    \draw (S1) -- (-3,2.5) node [left] {$S_{r_2}$};
    \draw (S2) -- (-3,-1.5) node [left] {$S_{r_1}$};
  \end{tikzpicture}
  \caption{The inner circle is $S_{r_1}$ and the outer circle is $S_{r_2}$. The horizontal dashed line is the hyperplane~$x_1=0$ and the vertical dashed line is the half-hyperplane~$x_i=0$, $x_1\geq 0$. The protected sets $Q^+$,~$Q^-$, and~$T$ are subsets of the grey regions shown. They must form at least~three connected components of protected sites because they are separated by the sites on the dashed lines.}\label{fi:3comps}
\end{figure}
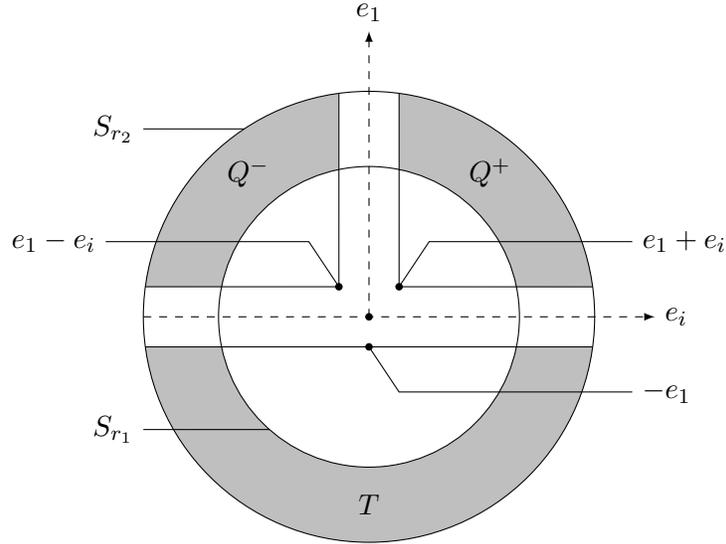

By construction, the three sets of protected sites $Q^+$,~$Q^-$, and~$T$ are mutually disconnected, so there are at least~three components of protected sites that intersect $P(S_{r_1 + 3d/2})$, contradicting Lemma~\ref{le:components}. This proves that $e_1$ does not have an opposing pair of protected neighbours in any direction except the first. A similar argument applies to~$-e_1$. Now, because each of $e_1$~and~$-e_1$ must have at least~$d + 1$ protected neighbours in total, it follows that each must have a protected out-neighbour in the first coordinate; that is, both $2e_1$ and $-2e_1$ are protected. Applying the same reasoning iteratively, we conclude that for all~$s$ with $s \leq r_2 - d$, both $se_1$ and $-se_1$ are protected.

Next, we show that the origin is 1-oriented. Assume for the sake of contradiction that there exists $i \neq 1$ such that both $e_i$ and $-e_i$ are protected. Let $C$, $C' \in \mathcal{C}_d$ denote the positive and negative $i$-child of the configuration~$(0, \dots, 0)$ respectively. Then, applying Lemma~\ref{le:key} to~$e_i$ with $C$ and to~$-e_i$ with $C'$, we find that $S_r$ contains at least~$2^{d-1}$ protected vertices~$y$ with $y_i \geq 1$ and at least~$2^{d-1}$ protected vertices~$y$ with $y_i \leq -1$, for all~$r_1+1\leq r\leq r_2$. However, the minimality of~$S_r$ means that these are the only protected vertices in $S_r$, and hence $S_r$ contains no protected vertices~$y$ with $y_i = 0$. This contradicts the fact that $re_1$ is protected. So without loss of generality, let $P_1 = \{e_1, -e_1, e_2, \dots, e_d\}$.

We continue to build the column inductively by showing that every vertex in the column with $k$ ones and $d-k-1$ zeros among its last $d-1$ coordinates is protected. This is proved in Claim~\ref{cl:stabind}, which takes up most of the remainder of the proof. Once we have the claim, we observe that it follows immediately that $P_r$ is canonical for all~$r_1\leq r\leq r_2-d$, and we note that one can follow the proof of Theorem~\ref{th:configs} to show that $P_{r_2-d+1}, \dots, P_{r_2-1}$ are canonical and that $P_{r_2}$ is semi-canonical.

We say that $x \in S_k$ is \emph{$k$-canonical} if $x_1 = 0$ and $x_i$ is either $0$~or~$1$ for all~$i>1$.

\begin{claim}\label{cl:stabind}
Let $k \geq 1$ and let $x \in S_k$ be a $k$-canonical site. Then $x$ is protected and 1-oriented, and $se_1 + x$ and $-se_1 + x$ are protected for all~$s\leq r_2-d$.
\end{claim}

\begin{remark}\label{re:kcanonsupport}
As noted above, we prove Claim~\ref{cl:stabind} by induction on $k$.  Fix $k \geq 1$ and suppose that for all~$\ell \leq k$, all $\ell$-canonical vertices are protected.  This means that if $x$ is $k$-canonical and $C \in \mathcal{C}_d$ is such that $\NN(C) = \emptyset$ and $\pos(C) \subset \{j \in [d] : x_j = 1\}$, then $x$ is $C$-supported. This observation will allow us to apply Lemma~\ref{le:minCcompat} at several points in the proof of Claim~\ref{cl:stabind}.
\end{remark}

\begin{proof}[Proof of Claim~\ref{cl:stabind}]
Throughout this proof we take $r$ to be an appropriate radius. To ensure that $S_r$ is minimal and that the hypotheses of Lemma~\ref{le:minCcompat} are satisfied, we set $r = \max\{M+d, r_1\}$, where $M$ is the maximum modulus of any site that we are currently considering.

We have already shown that all $1$-canonical vertices are protected. Let us show that $se_1+e_i$ and $-se_1+e_i$ are protected for all $s\leq r_2-d$ and~$i>1$. We know that $se_1$ does not have any opposing protected out-neighbours, so the claim is that for different values of~$s$ the out-neighbours have the same orientation. Suppose that this is false, and without loss of generality suppose that $se_1-e_2$ is protected. By applying Lemma~\ref{le:key} to~$se_1-e_2$ with $C = (1,-1,0,\dots,0)$, it follows that $S_r$ contains at least~$2^{d-2}$ protected sites~$y$ with $y_1 \geq r$ and $y_2 \leq -1$. Next, by applying Lemma~\ref{le:key} to~$e_2$ with $C' = (1,1,0,\dots,0)$, it follows that $S_r$ contains at least~$2^{d-2}$ protected sites~$z$ with $z_1 \geq 0$ and $z_2 \geq 1$. Let $C''$ denote the common parent of $C$~and~$C'$.  We have assumed without loss of generality that $-e_1 \in  P_1$, which means that the origin is $C''$-supported. Hence, by the minimality of~$S_r$, Lemma~\ref{le:minCcompat} applied to the origin says that there are exactly~$2^{d-1}$ protected sites~$w$ in $S_r$ with $w_1\geq 0$. Therefore there are no protected sites~$w$ in $S_r$ with $w_2=0$, contradicting the fact that $re_1$ is protected.

Next, we must show that the $1$-canonical vertices are $1$-oriented. First, we show that $2e_i$ is not protected for any $i>1$. Suppose without loss of generality that $2e_2$ is protected. Apply Lemma~\ref{le:key} to~$2e_2$ with $C=(0,1,0,\dots,0)$ to obtain $2^{d-1}$ protected vertices~$y$ in $S_r$ with $y_2\geq 2$. Then apply Lemma~\ref{le:key} again, this time to the origin with $C'=(0,-1,0,\dots,0)$, to obtain $2^{d-1}$ vertices~$z$ in $S_r$ with $z_2<0$. By minimality, together these sites form all of the protected sites in $S_r$, contradicting the fact that $(r-1)e_1+e_2$ is protected. Second, we show that we never have $e_i+e_j$ and $e_i-e_j$ both protected for distinct $i,j>1$. Suppose that this is false, and without loss of generality suppose that $e_2+e_3$ and $e_2-e_3$ are both protected. We apply Lemma~\ref{le:key} to~$e_2+e_3$ with $C=(0,1,1,0,\dots,0)$ and then to~$e_2-e_3$ with $C'=(0,1,-1,0,\dots,0)$ to obtain two disjoint sets of~$2^{d-2}$ protected sites~$y$ in $S_r$, all with $y_2\geq 1$ and $y_3\neq 0$. Define $C''$ to be the common parent of $C$~and~$C'$ and note that $e_2$ is $C''$-supported.  By applying Lemma~\ref{le:minCcompat} to~$e_2$, there are a total of exactly~$2^{d-1}$ protected sites~$y$ in $S_r$ with $y_2\geq 0$. But $(r-1)e_1+e_2$ is also protected, and this is a contradiction.

We have now proved the claim for $k = 1$. Proceeding by induction on $k$, let $x$ be a $k$-canonical site. First, we show that $x$ is protected. Then we show that $se_1+x$ and $-se_1+x$ are protected for all~$s\leq r_2-d$. Finally, we show that $x$ is $1$-oriented.

We begin by showing that $x$ is protected. Suppose not and choose any $i$ such that $x_i=1$. Then $x-e_i$ is $(k-1)$-canonical, and hence $1$-oriented. It has at least~$d+1$ protected neighbours, of which only $(x-e_i)+e_1$ and $(x-e_i)-e_1$ are opposing. Therefore, it has exactly~one protected neighbour in each of the last $d-1$ coordinates, and exactly~$k-1$ of those are in-neighbours. Therefore, $x-2e_i$ is protected. Since $k\geq 2$, we can also choose $j\neq i$ such that $x_j=1$. Notice that $x-e_j$ is also $(k-1)$-canonical. Now, define a configuration~$C$ by
\begin{equation*}
C(\ell) =
	\begin{cases}
	x_{\ell}, & \ell \neq j, \\
	0, 				& \ell = j.
	\end{cases}
\end{equation*}
Let $C'$ be the $i$-polar opposite of~$C$ and let $C''$ be their common parent. Apply Lemma~\ref{le:key} to~$x-e_j$ with $C$ to obtain $2^{d-k+1}$ protected sites in $S_r$ which are $C$-compatible with~$x-e_j$. Next, apply Lemma~\ref{le:key} to~$x-2e_i$ with $C'$ to obtain $2^{d-k+1}$ protected sites in $S_r$ which are $C'$-compatible with~$x-2e_i$. These two sets of sites are disjoint and all have $i$th coordinate equal to~zero, and furthermore, since $C''$ is the common parent of $C$~and~$C'$, they are all $C''$-compatible with~$x - e_i - e_j$. Moreover, by the induction hypothesis and Remark~\ref{re:kcanonsupport}, $x - e_i - e_j$ is $C''$-supported.  Hence, by Lemma~\ref{le:minCcompat} there are a total of exactly~$2^{d-k+2}$ protected sites in $S_r$ which are $C''$-compatible with~$x - e_i - e_j$, so we have found all of them. But $re_1$ is in $S_r$, is protected, and is $C''$-compatible with the origin, and it is not among our sites, which is a contradiction.

Next, we must show that $se_1+x$ and $-se_1+x$ are protected for all~$s\leq r_2-d$. The argument is almost identical to the one in the previous paragraph. Suppose that $se_1+x$ is not protected and let $i>1$ be such that $x_i=1$. Then because $se_1+x-e_i$ is protected and $1$-oriented, the site~$se_1+x-2e_i$ is protected. Define a configuration~$C$ by $C(j)=x_j$ for all~$j$, let $C'$ be the $i$-polar opposite of~$C$, and let $C''$ be their common parent. Applying Lemma~\ref{le:key} to~$x$ with $C$ and to~$se_1+x-2e_i$ with $C'$ gives two disjoint sets of~$2^{d-k}$ protected vertices in $S_r$ which are $C''$-compatible with~$x - e_i$, the first with $i$th coordinate at least~$1$, and the second with $i$th coordinate at most~$-1$. Lemma~\ref{le:minCcompat} says that there are exactly~$2^{d-k+1}$ protected vertices in $S_r$ which are $C''$-compatible with~$x - e_i$. However, the protected site~$re_1$ is in $S_r$ and is $C''$-compatible with~$x - e_i$, and it has $i$th coordinate~zero, a contradiction.

Finally, we show that $x$ is $1$-oriented. Again, the argument is almost identical to before. Suppose for some $i \neq 1$ that both $x + e_i$ and $x - e_i$ are protected. First, suppose further that $x_i=1$. Define a configuration~$C$ by $C(j)=x_j$ for all~$j$, let $C'$ be the $i$-polar opposite of~$C$, and let $C''$ be their common parent. Applying Lemma~\ref{le:key} to~$x+e_i$ with $C$ and to~$x-e_i$ with $C'$, we obtain two disjoint sets of~$2^{d-k}$ protected vertices in $S_r$ which are $C''$-compatible with~$x$, in the first case with $i$th coordinate at least~$2$ and in the second with $i$th coordinate at most~$0$. By minimality, Lemma~\ref{le:minCcompat} says that there are exactly~$2^{d-k+1}$ protected sites in $S_r$ which are $C''$-compatible with~$x$. But $(r-1)e_1+e_i$ is also protected, which is a contradiction.

Second, we suppose instead that $x_i=0$. This time, we define the configuration~$C$ by $C(i)=1$ and $C(k)=x_k$ otherwise, and, as usual, $C'$ is the $i$-polar opposite of~$C$ and $C''$ is their common parent. Then apply Lemma~\ref{le:key} to~$x+e_i$ with $C$ and to~$x-e_i$ with $C'$ to obtain two disjoint sets of~$2^{d-k-1}$ protected vertices in $S_r$ which are all $C''$-compatible with~$x$, and obtain a contradiction from Lemma~\ref{le:minCcompat} and the protected site~$re_1$. This proves the claim.
\end{proof}

Now we shall show that $P_r$ is canonical for all~$r \leq r_1 - 1$.  Let $r = r_1 - 1$.  If $x$ is not of the form~$\pm (r - k)e_1 + y$, where $y$ is a $k$-canonical vertex, then the fact that $P_{r_1}$ is canonical means that $x$ has no protected out-neighbours, which means that $x$ is not protected.  Thus, $P_{r_1 - 1}$ is canonical.  Iterating this argument shows that for all~$r \leq r_1 - 1$, $P_r$ is canonical.

We have proved that for all~$r\leq r_2-d$, $P_r$ is canonical.  To show that $P_r$ is canonical for $r = r_2-d+1,\dots,r_2-1$, we imitate the proof of Theorem~\ref{th:configs}. Since layers $r_2 - d + 1$, $\ldots\,$,~$r_2$ are all minimal, it follows by induction that $P_{r_2-d+1},\dots,P_{r_2-1}$ must be canonical and that $P_{r_2}$ must be semi-canonical.
\end{proof}

\section{Proofs of main results}\label{se:proofs}

Now that we have all of the necessary combinatorial tools, we start building up to the proofs of Theorems \ref{th:crittime} and~\ref{th:conc}. Let $E_t(x)$ be the event that a site~$x$ is uninfected at time~$t$, and let $F_t(x)$ be the indicator random variable for $E_t(x)$. The sequence of random variables that we are interested in is $(F_t(n))_{n=1}^\infty$, where $F_t(n) = \sum_{x \in V(\T_n^d)} F_t(x)$. The mean of~$F_t(n)$ is $\E F_t(n):=\lambda_n:=n^d \rho_1$, where
\[
\rho_1 = \P_{p_n}\bigl(E_t(x)\bigr).
\]

Most of this section is devoted to proving the following Poisson convergence result, from which Theorems \ref{th:crittime} and~\ref{th:conc} will follow easily.   Because we are mainly interested in uninfected sites, rather than infected sites, we shall often work with $q = 1 - p$ instead of with $p$.

\begin{theorem}\label{th:poisson}
Let $t=o(\log n/\log\log n)$ and let $p_n$ be such that $q_n = 1 - p_n \leq C n^{-d/m_{t,d}}$. Then
\[
\dTV\bigl(F_t(n), \Po(\lambda_n)\bigr) = O\bigl(t^d q_n \bigr) = o(1).
\]
\end{theorem}

Our first task is to estimate $\rho_1$. To do this, we make use of the stability result of the previous section to bound the number of configurations of~$m_t+k$ uninfected sites that protect a given site.

For the variance, we shall need to estimate the probability that both $x$ and $y$ are uninfected at time~$t$ when $x$ and $y$ are close enough for these events to be dependent. For this, first we need a lemma which says that $m_t$ uninfected sites are not enough to protect two distinct sites. In other words, if $x$ and $y$ are distinct protected sites then $\abs{P(B_t(x))\cup P(B_t(y))} \geq m_t + 1$. We then use this together with the stability result to bound the quantity
\[
\rho_2 = \max \, \big\{\P_{p_n}\big(E_t(x)\cap E_t(y)\big) \, : \, \lVert x-y\rVert \leq 2t\big\}.
\]

Once we have these bounds on $\rho_1$ and $\rho_2$, the proof of Theorem~\ref{th:poisson} will be just a few lines.

Throughout this section, all constants, either explicit or implied by the $O(\cdot)$ notation, will be quantities that depend only on~$d$.

\begin{lemma}\label{le:ratiobound}
Let $t=o(\log n/\log\log n)$ and let
\begin{equation}\label{eq:pbound}
q = 1 - p \leq C n^{-d/m_{t,d}}
\end{equation}
for some $C>0$. Then for any constant~$c>0$,
\[
t^c q \leq \exp\bigl(-\Omega(\log\log n)\bigr).
\]
\end{lemma}

We shall only ever need the corollary $t^c q = o(1)$.

\begin{proof}
By~\eqref{eq:pbound} we have
\[
\log q = \log(1-p) \leq \log C - \dfrac{d}{m_{t,d}}\log n.
\]
It follows that
\[
t^c q \leq \exp\Bigl( c\log t - \dfrac{d}{m_{t,d}}\log n + \log C \Bigr).
\]
Let $t=\log n/\omega(n) \log\log n$ for some function~$\omega(n)\rightarrow\infty$. By Corollary~\ref{MinProtected}, $m_{t,d} \leq t2^d$. Putting these into the last inequality we obtain
\[
t^c q \leq \exp\Bigl( c\log\log n - \dfrac{d}{2^d} \omega(n) \log\log n + \log C \Bigr),
\]
which is certainly enough to prove the lemma.
\end{proof}

We could have replaced the constant~$C$ in the above lemma by a function as large as~$\log n$, but we shall not need that in the applications that follow.

Now we determine up to a factor of~$1+o(1)$ the probability $\rho_1$ that a site is uninfected at time~$t$.

\begin{theorem}\label{th:rho1bound}
Let $t=o(\log n/\log\log n)$ and let $p$ satisfy~\eqref{eq:pbound}. Then
\begin{equation}\label{eq:qest}
\rho_1 = \bigl(1+o(1)\bigr) d^3 2^{d-1} q^{m_{t,d}}.
\end{equation}
\end{theorem}

\begin{proof}
We define $g_t(k)$ to be the number of arrangements of~$m_t+k$ uninfected sites in $B_t$ such that the origin is protected. Summing over $k$ we obtain
\begin{equation}\label{eq:failbound}
\rho_1 = \sum_{k=0}^{\abs{B_t} - m_t} g_t(k) p^{\abs{B_t} - m_t - k} q^{m_t + k}.
\end{equation}

We need to bound $g_t(k)$. The stability theorem, Theorem~\ref{th:stability}, says that if there are $3d + 1$ consecutive minimal layers, then the uninfected sites in these layers are part of a semi-canonical set. However, if there are at most~$3d$ consecutive minimal layers, then Lemma~\ref{le:components}, and hence Theorem~\ref{th:stability}, does not hold, and the results of Section~\ref{se:extremal} do not tell us anything about the structure of the uninfected sites in these layers.

There are at most~$k$ non-minimal layers. In the worst case, there are exactly~$k$ non-minimal layers, and they are all far apart. In this case, we place uninfected vertices in each of these layers arbitrarily, as well as in the $3d$ layers following each non-minimal layer. This means that we have placed uninfected sites arbitrarily in at most~$(3d + 1)k$ layers. There are at most~$2^d (3d + 1)k + k$ total uninfected sites in these layers. Each layer has at most~$\abs{S_t} \leq c_1t^{d - 1}$ vertices, so the number of ways of placing the uninfected sites is at most
\[
\dbinom{c_1 t^{d - 1}}{2^d(3d + 1)k + k} \leq (c_1 t^{d - 1})^{2^d(3d + 1)k + k} = t^{O(k)}.
\]
It is important here that the exponent on the right-hand side does not depend on~$t$.

All of the layers whose uninfected vertices we have not yet placed are minimal and are contained in bands of at least~$3d + 1$ consecutive minimal layers. Consider the outer-most band; say this is the range $r_1\leq r\leq r_2$. By Theorem~\ref{th:stability}, $P(B_{r_2})$ must be semi-canonical. In particular, if $r<r_2$ is such that $S_r$ is minimal, then $P(S_r)$ is canonical with fixed alignment and orientations. Hence, by Corollary~\ref{NumExtConfigs}, there are at most~$d^3 2^{d-1}$ ways to place the rest of the uninfected sites. In fact, all that we shall use is that this quantity is~$O(1)$. We have thus shown that
\begin{equation}\label{eq:gtkBound}
g_t(k) = O\bigl(t^{O(k)}\bigr).
\end{equation}
Putting this together with~\eqref{eq:failbound} gives
\begin{align*}
\rho_1 &= g_t(0)p^{\abs{B_t}-m_t}q^{m_t} \Bigg( 1 + \sum_{k=1}^{\abs{B_t}-m_t} \dfrac{g_t(k)}{g_t(0)} p^{-k} q^k \Bigg) \\
&= g_t(0)p^{\abs{B_t}-m_t}q^{m_t} \Bigg( 1 + O\Bigg(\sum_{k=1}^{\abs{B_t}-m_t} t^{O(k)} q^k \Bigg) \Bigg),
\end{align*}
where here we have used $p>1/2$ and absorbed the $2^k$ term into $t^{O(k)}$. Hence by Lemma~\ref{le:ratiobound},
\[
\rho_1 = g_t(0)p^{\abs{B_t}-m_t}q^{m_t}\bigl(1+o(1)\bigr).
\]
It remains to show that $p^{\abs{B_t}-m_t}=1-o(1)$. This also follows from Lemma~\ref{le:ratiobound}, because $\abs{B_t}-m_t\geq t^c$ for some $c$, so the problem is equivalent to showing that $t^c\log p = t^c\log(1 - q) = o(1)$. Hence,
\[
\rho_1 = \bigl(1+o(1)\bigr) d^3 2^{d-1} q^{m_{t,d}}. \qedhere
\]
\end{proof}

The next two lemmas give a bound on $\rho_2$, which we defined as
\[
\rho_2 = \max \bigl\{\P_{p_n}\bigl(E_t(x)\cap E_t(y)\bigr) \, : \, \lVert x-y\rVert \leq 2t\bigr\}.
\]

\begin{lemma}\label{le:minunion}
Suppose that $x$, $y \in \T_n^d$ are both protected. Then $B_t(x) \cup B_t(y)$ contains at least~$m_t + 1$ uninfected sites.
\end{lemma}

\begin{proof}
By translation invariance we may assume that $x=0$. The result is trivial unless $\norm{y} \leq t$, so we suppose that this is the case and that $B_t(0)$ and $B_t(y)$ are minimal. By Corollary~\ref{MinProtected} it suffices to show that $P(B_t(0))\neq P(B_t(y))$, or equivalently that a semi-canonical set cannot protect two distinct vertices. Assume that $P(B_t)$ is 1-oriented and that $\varepsilon_2 = \ldots = \varepsilon_d = 1$, and suppose that $P(B_t)$ protects a site~$z$. If either extreme point of~$P(B_t)$ is $te_1$ or $-te_1$ then we are forced to take $z=0$. Otherwise, we have a site in $P(B_t)$ of the form~$(t-1)e_1-e_2$, say, and $(t-2)e_1-e_2$ not in $P(B_t)$. This again forces $z=0$.
\end{proof}

\begin{lemma}\label{le:pxyBound}
Let $t=o(\log n/\log\log n)$ and let $p$ satisfy~\eqref{eq:pbound}. Then
\[
\rho_2 = O(\rho_1 q).
\]
\end{lemma}

In the applications all that we shall use is that $\rho_2=o(\rho_1)$.

\begin{proof}
Let $x$ and $y$ be sites in $\T^d$ such that $\lVert x-y\rVert \leq 2t$. If $E_t(x)\cap E_t(y)$ occurs then Lemma~\ref{le:minunion} says that $B_t(x) \cup B_t(y)$ contains at least~$m_t + 1$ uninfected sites. Let $h_t(k)$ denote the number of configurations of~$m_t + 1 + k$ uninfected sites in $B_t(x) \cup B_t(y)$ such that both $x$ and $y$ are protected. Thus,
\begin{equation}\label{eq:pxyBound}
\rho_2 \leq \sum_{k=0}^{2\abs{B_t}} h_t(k) q^{m_t+1+k},
\end{equation}
using the bound $p\leq 1$.

We count the number of valid configurations such that $B_t(x) \setminus B_t(y)$ contains exactly~$i$ uninfected sites, $B_t(y) \setminus B_t(x)$ contains exactly~$j$ uninfected sites, and $B_t(y) \cap B_t(x)$ contains exactly~$\ell$ uninfected sites, where
\[
i + j + \ell = m_t + 1 + k.
\]
For each such choice of $i$,~$j$, and~$\ell$, we bound the number of valid configurations from above by placing $i + \ell$ uninfected sites in $B_t(x)$ and $j + \ell$ uninfected sites in $B_t(y)$ independently. Thus,
\[
h_t(k) \leq \sum_{(i,j,\ell)} g_t(i+\ell-m_t) g_t(j+\ell-m_t),
\]
where the sum is over valid triples~$(i,j,\ell)$. Very crudely, there are at most
\[
(m_t+1+k)^3 = t^{O(1)}
\]
triples, using $m_t \leq t2^d$ from Corollary~\ref{MinProtected}. Using the bound on $g_t(k)$ from~\eqref{eq:gtkBound}, it follows that
\[
h_t(k) = O\big( t^{O(i+\ell-m_t+j+\ell-m_t)} \big).
\]
We have the trivial bound $\ell \leq m_t + 1 + k$, so we can simplify this expression to
\[
h_t(k) = O\big(t^{O(k)}\big).
\]
Combining this with~\eqref{eq:pxyBound} gives the following bound on $\rho_2$.
\[
\rho_2 = O(1) \sum_{k=0}^{2\abs{B_t}} t^{O(k)} q^{m_t+1+k}.
\]
Now we proceed as we did in Theorem~\ref{th:rho1bound} to estimate $\rho_1$. By Lemma~\ref{le:ratiobound}, we can say that the first term in this last equation dominates the sum, so we have
\[
\rho_2 = O\left( q^{m_t+1} \right) = O(\rho_1 q),
\]
using the estimate of~$\rho_1$ from~\eqref{eq:qest} for the last equality.
\end{proof}

Finally, we are ready to use the Stein-Chen method to prove Theorem~\ref{th:crittime}. For each $x \in \T_n^d$, let $N_x=B_{2t+1}(x)$ be the dependency neighbourhood of~$x$. Observe that with this choice of~$N_x$, $F_t(x)$ is independent of~$\{F_t(y) : y \notin N_x\}$, as required.

\begin{proof}[Proof of Theorem~\ref{th:poisson}]
Using the Stein-Chen method (Theorem~\ref{th:steinchen}), we just have to show that
\[
\min \{1,\lambda_n^{-1}\} \Biggl( \sum_{x\in\T_n^d}\sum_{y\in N_x} \rho_1^2 + \sum_{x\in\T_n^d} \sum_{y\in N_x\setminus\{x\}} \rho_2 \Biggr) \rightarrow 0,
\]
or, equivalently, that
\[
\min \{1,\lambda_n^{-1}\} n^d t^d ( \rho_1^2 + \rho_2 ) \rightarrow 0.
\]
Since $\lambda_n = n^d \rho_1$, the left-hand side is at most
\[
t^d \biggl( \rho_1 + \dfrac{\rho_2}{\rho_1} \biggr).
\]
Using Lemma~\ref{le:pxyBound}, which was the bound on $\rho_2$, this is
\[
t^d \rho_1 + t^d O(q_n).
\]
By the bound on $\rho_1$ from Theorem~\ref{th:rho1bound}, we have $\rho_1 = o(q_n)$. This and Lemma~\ref{le:ratiobound} mean that we may write the above expression as
\[
t^d \rho_1  +  t^d O(q_n) = O\bigl(t^d q_n\bigr) = o(1),
\]
which completes the proof.
\end{proof}

\begin{proof}[Proof of Theorem~\ref{th:crittime}]
By a standard coupling argument, the probability of percolating in time at most~$t$ is increasing in $p$. Therefore, if
\[
\liminf_{n\to\infty}\, (1 - p_n) n^{d/m_{t,d}} = \liminf_{n\to\infty} q_n n^{d/m_{t,d}} = \infty,
\]
then the result follows from the result for larger values of~$p$. Hence, we may assume that $p$ satisfies the usual bound,~\eqref{eq:pbound}, and that Theorem~\ref{th:poisson} applies. This tells us that $F_t(n)$ converges in distribution to~$\Po(\lambda_n)$, so that
\[
\P_{p_n}(F_t(n) = 0) = \bigl(1+o(1)\bigr) e^{-\lambda_n}.
\]
The mean~$\lambda_n$ is equal to~$n^d \rho_1$, and the estimate of~$\rho_1$ from Theorem~\ref{th:rho1bound} gives
\[
\rho_1 = \Theta(1) q_n^{m_{t,d}}.
\]
Therefore,
\begin{equation*}
\P_{p_n}(F_t(n) = 0) \to
	\begin{cases}
	1 & \text{if } \lim_{n\to\infty} q_n n^{d/m_{t,d}} = 0, \\
	0 & \text{if } \lim_{n\to\infty} q_n n^{d/m_{t,d}} = \infty,
	\end{cases}
\end{equation*}
as we wanted.

Finally, given $\alpha \in (0, 1)$, to determine $p_{\alpha}(t)$, simply observe that $\alpha \sim e^{-\lambda_n} \sim \exp(-d^3 2^{d-1} n^d q_n^{m_t})$ and solve for $p_n$.  This completes the proof of Theorem~\ref{th:crittime}.
\end{proof}

Theorem~\ref{th:conc} now follows easily.

\begin{proof}[Proof of Theorem~\ref{th:conc}]
Suppose that for all~$n$, $q_n$ satisfies~\eqref{eq:1ptrange}.  Then, by Theorem~\ref{th:crittime}, $\P_{p_n}(T \leq t - 1) = o(1)$ and $\P_{p_n}(T \geq t + 1) = o(1)$, which proves part~(i).

Suppose instead that for all~$n$, $q_n$ satisfies~\eqref{eq:2ptrange}.  Then $q_n \geq \bigl(n^{-d} \omega(n)\bigr)^{1/m_{t-1}}$, so, by Theorem~\ref{th:crittime}, we have $\P_{p_n}(T \leq t - 1) = o(1)$.  Similarly, $\P_{p_n}(T \geq t + 2) = o(1)$, so, with high probability, $T \in \{t, t + 1\}$.  Now suppose that $\lim_{n \to \infty} q_n^{m_t} n^d = c$.  Then
\[
\P_{p_n}(T = t) \sim \P_{p_n}(T \leq t) \sim	e^{-\lambda_n} \sim \exp\Bigl(-d^3 2^{d-1} n^d q_n^{m_t}\Bigr) \sim \exp\bigl(-d^3 2^{d-1} c\bigr).
\]
By a similar argument, we have $\P_{p_n}(T = t + 1) \sim 1 - \exp(-d^3 2^{d-1} c)$, which proves part~(ii).
\end{proof}

\section{The modified \texorpdfstring{$d$}{d}-neighbour model}\label{se:mod}

As noted in Section~\ref{se:fastintro}, the proof of Theorem~\ref{th:crittimemod} is very similar to that of Theorem~\ref{th:crittime}, so we shall only sketch the argument.

\begin{proof}[Sketch of proof of Theorem~\ref{th:crittimemod}]
We shall show that if the origin is protected under the modified $d$-neighbour model, then $\abs{P(B_t^d)} \geq 2t + 1$.  We shall also show that the only minimal configurations are columns centred at the origin.

First, we observe that if the origin is protected and there exists $r \geq 1$ such that $S_r$ contains only one protected site, then some hemisphere (without loss of generality, the set~$\{x \in B_t^d : x_1 \geq 0\}$) contains no protected sites, and the origin becomes infected by time~$r \leq t$, a contradiction.  By the same reasoning, if a layer contains exactly~two protected sites, then they must be antipodal points, that is, they must be of the form~$\pm te_i$ for some~$i \in [d]$.  Second, we observe that if a vertex is protected, then it must have an opposing pair of protected neighbours.  Combining these two observations shows that if the origin is protected and $\abs{P(B_t^d)} = 2t + 1$, then the protected sites must form a column centred at the origin.  Clearly, there are~$d$ such columns.

In this case, the analogue of the stability result (Theorem~\ref{th:stability}) is trivial.  Much as in the case of the standard $d$-neighbour model, we say that a sphere~$S_r$ is \emph{minimal} if $\abs{P(S_r)} = 2$ and that a ball~$B_r$ is \emph{minimal} if $\abs{P(B_r)} = 2r + 1$.

\begin{lemma}\label{le:modstability}
Suppose that the origin is protected and that for some $1 \leq r \leq t$, the sphere~$S_r$ is minimal.  Then $P(B_r)$ is a column of height~$2r + 1$ centred at the origin.
\end{lemma}

\begin{proof}
Let $S_r$ be a minimal layer.  As shown above, the protected sites in $S_r$ must be of the form~$\pm re_i$ for some $i \in [d]$.  If $x \in S_{r-1}$ is any site besides~$\pm (r - 1)e_i$, then $x$ has no protected out-neighbours, which means that $x$ is not protected.  Thus, the only protected sites in $S_{r-1}$ are $(r - 1)e_i$ and $-(r - 1)e_i$.  Iterating this argument shows that $P(B_r)$ is a column centred at the origin.
\end{proof}

Let $g_t(k)$ denote the number of arrangements of~$2t + 1 + k$ uninfected sites in $B_t$ such that the origin is protected.  We bound $g_t(k)$ from above as follows.  There are at most~$k$ non-minimal layers in $B_t$, which means that there are a total of at most~$3k$ uninfected sites in these layers.  We place uninfected vertices in each of these layers arbitrarily.  Each such layer contains $O(t^{d-1})$ vertices, so the number of ways of placing the uninfected sites in these layers is at most
\[
\dbinom{c_1 t^{d-1}}{3k} = t^{O(k)}.
\]
Note that, as in the case of the standard $d$-neighbour model, the exponent on the right-hand side does not depend on~$t$.  Now we turn to the uninfected sites in the minimal layers.  By Lemma~\ref{le:modstability}, if any layer~$S_r$ is minimal, then $P(B_r)$ is a column of height~$2r + 1$ centred at the origin.  There are $d$ choices for this column.  Hence
\[
g_t(k) = O\bigl(t^{O(k)}\bigr).
\]

Next, it is easy to see that if $x$ and $y$ are distinct protected vertices, then $B_t(x) \cup B_t(y)$ contains at least~$2t + 2$ uninfected vertices, because if $B_t(x)$ is minimal, then the column that protects $x$ cannot also protect $y$.

Finally, because $2t + 1$ is linear in $t$, the same argument as in the proof of Lemma~\ref{le:ratiobound} shows that if $q = 1 - p \leq Cn^{-d/(2t + 1)}$ for some $C > 0$, then for all~$c > 0$, we have $t^c q = o(1)$ as $n \to \infty$.  The rest of the proof of Poisson convergence then follows as in Section~\ref{se:proofs}.  Indeed, letting $F_t(n)$ denote the number of sites that are uninfected at time~$t$ and setting $\mu_n = \E F_t(n)$, it follows that
\[
\mu_n = \bigl(1 + o(1)\bigr) d n^d q_n^{2t + 1}
\]
and that if $q_n = 1 - p_n \leq Cn^{-d/(2t + 1)}$ for some $C > 0$, then
\[
\dTV\bigl(F_t(n), \Po(\mu_n)\bigr) = O\bigl(t^d q_n\bigr) = o(1),
\]
as we wanted.
\end{proof}

\section{Possible generalizations and conjectures}\label{se:open}

\emph{Other thresholds.} It is possible to generalize the results of this paper to $r$-neighbour bootstrap percolation in $d \geq 2$ dimensions for all~$2 \leq r \leq d$. Call a subset~$X$ of~$B_t^d$ \emph{$(d, r)$-canonical} if there exist $j_1,\dots,j_{r-1}\in[d]$ and for each $i\in\{j_1,\dots,j_{r-1}\}$ an orientation~$\varepsilon_i\in\{-1,1\}$ such that
\[
X = \bigl\{(x_1,\dots,x_d)\in B_t^d : x_i\in\{0,\varepsilon_i\} \text{ for all } i \in \{j_1,\dots,j_{r-1}\} \bigr\}.
\]
A $(d, d)$-canonical set is canonical and in general a $(d, r)$-canonical set is a union of~$2^{r-1}$ $(d - r + 1)$-dimensional affine subspaces intersected with~$B_t^d$.

Let $m_t(d,r)$ be the size of a $(d, r)$-canonical set in $\Z^d$ of radius~$t$. (So $m_t(d,d) = m_{t,d}$.) As usual, we let $P(X)$ denote the set of protected sites in $X$ and let $g_t(k)$ denote the number of configurations of~$m_t(d,r) + k$ sites in $B_t^d$ such that the origin is protected under $r$-neighbour bootstrap percolation.  The following claims are proved in~\cite{speedallr}.

\begin{claim}\label{cl:allrmin}
Let $t\geq 0$ and $d \geq r \geq 2$. Suppose that the origin is protected under $r$-neighbour bootstrap percolation. Then
\[
\lvert P(B_t^d) \rvert \geq m_t(d,r).
\]
Moreover, the number of configurations of protected sites which attain this bound does not depend on~$t$; and $g_t(k) = O(t^{O(k)})$, where the implicit constants depend only on $d$~and~$r$.
\end{claim}

\begin{claim}\label{cl:allrtime}
Let $d \geq r \geq 2$, let $t = t(n) = o\left((\log n/\log\log n)^{1/(d-r+1)}\right)$, let $(p_n)_{n=1}^\infty$ be a sequence of probabilities, let $\omega(n) \to \infty$,  and let $T = T(\T_n^d)$.  Under the $r$-neighbour model,
\begin{enumerate*}
\item if $q_n \leq \bigl(n^{-d}\!/\omega(n)\bigr)^{1/m_t(d,r)}$, then $\P_{p_n}(T \leq t) \to 1$ as $n\to\infty$;
\item if $q_n \geq \bigl(n^{-d} \omega(n)\bigr)^{1/m_t(d,r)}$, then $\P_{p_n}(T \leq t) \to 0$ as $n\to\infty$.
\end{enumerate*}
\end{claim}

\emph{Range of\/~$t$.} Theorem~\ref{th:crittime} gives the critical probability for percolation by time~$t$ for values of~$t$ up to~$o(\log n/\log\log n)$, or in the dual form, it gives a concentration result for the percolation time~$T$ for sequences of probabilities close to~$p_n = 1 - n^{-d/m_{t,d}}$ for some $t=o(\log n/\log\log n)$. Were the results to hold for $t$ as large as~$o(\log n)$, then this would give the percolation time for all probabilities in the range~$1-o(1)$. We conjecture that this should be the case.

\begin{conjecture}
Theorem~\ref{th:crittime} holds for all~$t$ in the range~$t=o(\log n)$.
\end{conjecture}

\emph{Other ranges of~$p$.} We have only looked at the percolation time for $p$ very close to~1. It is interesting to ask what one can say about the time for other values of~$p$.

\section{Acknowledgements}\label{se:ack}

We would like to thank an anonymous referee for helpful comments.

\bibliographystyle{amsplain}
\bibliography{Bootstrapbib,ProbBib,miscbib}

\end{document}